\documentclass[final,onefignum,onetabnum]{siamart171218}

\usepackage{braket,amsfonts}

\usepackage{array}
\usepackage{amsmath}               
  {
      \theoremstyle{plain}
      \newtheorem{assumption}{Assumption}
  }
\usepackage[caption=false]{subfig}
\usepackage{pgfplots}

\newsiamthm{claim}{Claim}
\newsiamremark{remark}{Remark}
\newsiamremark{hypothesis}{Hypothesis}
\crefname{hypothesis}{Hypothesis}{Hypotheses}

\usepackage{algorithmic}

\usepackage{graphicx,epstopdf}

\Crefname{ALC@unique}{Line}{Lines}

\usepackage{amsopn}

\usepackage{xspace}
\usepackage{bold-extra}
\usepackage[most]{tcolorbox}

\colorlet{texcscolor}{blue!50!black}
\colorlet{texemcolor}{red!70!black}
\colorlet{texpreamble}{red!70!black}
\colorlet{codebackground}{black!25!white!25}

\newcommand{\be}{\begin{equation}}
    \newcommand{\ee}{\end{equation}}
\newcommand{\bee}{\begin{equation*}}
    \newcommand{\eee}{\end{equation*}}
\newcommand{\bea}{\begin{eqnarray}}
    \newcommand{\eea}{\end{eqnarray}}
\newcommand{\beaa}{\begin{eqnarray*}}
    \newcommand{\eeaa}{\end{eqnarray*}}

\newcommand{\st}{\,\textrm{s.t.}\,}

\newcommand{\dist}{\mathrm{dist}}

\newcommand{\iprod}[2]{\left \langle #1, #2 \right \rangle }

\usepackage{enumerate}

\lstdefinestyle{siamlatex}{%
  style=tcblatex,
  texcsstyle=*\color{texcscolor},
  texcsstyle=[2]\color{texemcolor},
  keywordstyle=[2]\color{texemcolor},
  moretexcs={cref,Cref,maketitle,mathcal,text,headers,email,url},
}

\tcbset{%
  colframe=black!75!white!75,
  coltitle=white,
  colback=codebackground, 
  colbacklower=white, 
  fonttitle=\bfseries,
  arc=0pt,outer arc=0pt,
  top=1pt,bottom=1pt,left=1mm,right=1mm,middle=1mm,boxsep=1mm,
  leftrule=0.3mm,rightrule=0.3mm,toprule=0.3mm,bottomrule=0.3mm,
  listing options={style=siamlatex}
}

\newtcblisting[use counter=example]{example}[2][]{%
  title={Example~\thetcbcounter: #2},#1}

\newtcbinputlisting[use counter=example]{\examplefile}[3][]{%
  title={Example~\thetcbcounter: #2},listing file={#3},#1}

\DeclareTotalTCBox{\code}{ v O{} }
{ 
  fontupper=\ttfamily\color{black},
  nobeforeafter,
  tcbox raise base,
  colback=codebackground,colframe=white,
  top=0pt,bottom=0pt,left=0mm,right=0mm,
  leftrule=0pt,rightrule=0pt,toprule=0mm,bottomrule=0mm,
  boxsep=0.5mm,
  #2}{#1}

\patchcmd\newpage{\vfil}{}{}{}
	\title{A Stochastic Composite Augmented Lagrangian Method For Reinforcement Learning\thanks{Submitted to the editors DATE.
			\funding{Research supported in part by the NSFC grant 11831002 and by Beijing Academy of Artificial Intelligence (BAAI).}}}
	
	\author{Yongfeng Li\thanks{Beijing International Center for Mathematical Research, Peking University, CHINA (\email{yongfengli@pku.edu.cn}).}
	\and Mingming Zhao\thanks{Beijing International Center for Mathematical Research, Peking University, CHINA (\email{mmz102@pku.edu.cn}).}
	\and Weijie Chen\thanks{Academy for Advanced Interdisciplinary Studies, Peking University, CHINA(\email{baoz@pku.edu.cn}).}
\and Zaiwen Wen\thanks{Corresponding author. Beijing International Center for Mathematical Research, Peking University,CHINA}(\email{wenzw@pku.eu.cn}).}

	\headers{A Stochastic Composite ALM For Reinforcement Learning}{Yongfeng Li, Mingming Zhao,  Weijie Chen, Zaiwen Wen}

\graphicspath{ {./figure/} }

\ifpdf
\hypersetup{ pdftitle={A Stochastic Composite Augmented Lagrangian Method For Reinforcement Learning} }
\fi

\begin{document}
	\maketitle
	
\begin{abstract}
In this paper, we consider the linear programming (LP) formulation for deep reinforcement learning. The number of the constraints depends on the size of state and action spaces, which makes the problem intractable in large or continuous environments. The general augmented Lagrangian method suffers the double-sampling obstacle in solving the LP. Namely, the conditional expectations originated from the constraint functions and the quadratic penalties in the augmented Lagrangian function impose difficulties in sampling and evaluation. Motivated from the updates of the multipliers, we overcome the obstacles in minimizing the augmented Lagrangian function by replacing the intractable conditional expectations with the multipliers. Therefore, a deep parameterized augment Lagrangian method is proposed. Furthermore, the replacement provides a promising breakthrough to integrate the two steps in the augmented Lagrangian method into a single constrained problem. A general theoretical analysis shows that the solutions generated from a sequence of the  constrained optimizations converge to the optimal solution of the LP if the error is controlled properly.  A theoretical analysis on the quadratic penalty algorithm under neural tangent kernel setting shows the residual can be arbitrarily small if the parameter in network and optimization algorithm is chosen suitably.  Preliminary experiments illustrate that our method is competitive to other state-of-the-art algorithms.

\end{abstract}


\section{Introduction}
With the recent development and empirical successes of the neural networks, deep reinforcement learning has been resoundingly applied in many industrial fields, such robotics and control, economics, automatic driving, etc. The focus of reinforcement learning (RL) is solving a Markov decision process (MDP) which models the framework for interactions between the agent and the environment. In real-world problems, the issues such as high dimensional or/and continuous state and action spaces, partially observed state informations, limited and noised samples, unknown dynamics and etc., make RL significantly harder than general deep learning tasks. 

A fundamental conclusion in RL is that the Bellman operator is contractive in value function space and the optimal value function is the unique fixed point of the Bellman operation \cite{dai2018sbeed}. It is not easy to solve the Bellman equation greedily in large state and action spaces or when the function parameterization is used because of the non-smoothness and highly non-convexity. Instead of formulating such a fixed point iteration, value-based RL algorithms concentrate on minimizing the Bellman residual, which is a square loss in general, to approximate the optimal value function with stochastic gradient descent methods. However, there are still two troublesome issues in this alternative formulation. One is the conditional expectation over the next state suffers from the double-sampling obstacle in gradient formulation. In other words, the objective function is a square of a conditional expectation of value functions, thus to construct the unbiased gradient estimation in a stochastic behavior, two independent sets of samples of the next transition state are required for each pair of state and action. However, only one single transition state sample is observed in sampling cases or model-free setting. This obstacle has oppressed the designs of RL algorithms in practice. The other is that the $\max$ operator in the optimal Bellman equation exposes the non-smoothness and difficulty of the optimization in complex environments. To overcome the former challenge, several methods are proposed in recent researches. Wang et al., \cite{wang2017stochastic} takes an additional variable to track the inner expectation and proposes a class of stochastic compositional gradient descent algorithms by using noisy sample gradients. Borrowing extra randomness from the future is proposed in \cite{zhu2020borrowing}. Conditioned on the assumption that the transition kernel can be described with a stochastic differential equation, they propose two algorithms in which the training trajectories are shown to be close to the one of the unbiased SGD. Deep Q-Network (DQN) \cite{mnih2015human,hessel2018rainbow} concentrates on the Bellman equation of the state and action value function and predicts the value for the next state with the target parameter which is periodically synchronized to the training parameters by minimizing the Bellman residual. Under the off-policy setting, DQN performs SGD and achieves more than human players level in Atari Games \cite{bellemare2013arcade,brockman2016openai}. As for the second problem that the $\max$ operator in Bellman equation may be unavailable or expensive, an additional actor network is introduced in \cite{fujimoto2018addressing,lillicrap2015continuous} to adapt to large or/and continuous environments. By adding the scaled log-policy to the one-step reward in \cite{vieillard2020munchausen}, the $\max$ operator is slightly converted into a soft formulation. A smoothed Bellman equation is introduced in \cite{dai2018sbeed} and it is reformulated as a novel primal-dual optimization problem. 

The optimal Bellman equation can be interpreted as a linear programming (LP) problem with respect to the state value function \cite{bertsekas1995dynamic,puterman2014markov}. The uniqueness of the fixed point of the Bellman operator guarantees the existence of the solution of the LP problem. The dual form of the LP problem can also be viewed as an interpretation of maximizing the cumulative reward in policy space. As the transition kernel of the MDP may be unavailable, such as in model-free setting, the dual problem is rarely considered since its constraint functions are hard to estimate. 
Consequently, the equivalency between the primal and dual LP problems, as well as their relationship to the optimal policy, provide a broad breakthrough for the design of RL algorithm.
By formulating the primal and dual LPs into an equivalent saddle-point problem, a stochastic primal-dual algorithm is proposed in \cite{wang2017randomized} to find a near optimal policy in tabular case. For large or even continuous state and action spaces, the bilinear structure of the Lagrangian function is destroyed because of the parameterization, the general convergence result for the parametrized saddle-point problem no longer holds due to the lack of convex-concavity. Furthermore, the large bias of sample estimations and under-fitted functions bring great disturbance to the stability of numerical performance.  To promote the local convexity and duality, a path regularization is added to the objective function in \cite{dai2017boosting} to restrict the feasible domain of the value function to be a ball guided by the behavior policy, and the general Bellman residual is extended to the multi-step version for a better bias-variance trade off. For sample efficiency, the authors exploit stochastic dual ascent method and solve a sequence of partial Lagrangian functions to update the value function for each policy. To overcome the path dependency and time consuming in \cite{dai2017boosting}, the path regularization is replaced with a quadratic penalty term with respect to the Bellman residual in \cite{cho2017deep} and a sequence of stochastic gradient updates for primal and dual variables are performed alternately.

In this work, we propose a novel stochastic optimization method for the primal LP problem derived from the algorithmic framework of the augmented Lagrangian method (ALM). The optimal policy can be recovered from a specific relationship with the multipliers \cite{cho2017deep,dai2017boosting}. Specifically, a weighted augmented Lagrangian function rather than the standard form  is constructed based on a weighted summation of the quadratic penalty terms defined by the multipliers and the constraints. The weights are functions on state and action spaces. They enable the sampling process and are not required in practice. 
As the constraint functions include conditional expectations, the quadratic terms induce a challenge to the calculation of the weighted augmented Lagrangian function, as well as a double sampling obstacle in its gradient estimation. 
Actually, it is claimed in the appendix of \cite{dai2017boosting} that minimizing the augmented Lagrangian function itself is quite difficult. 

Our main contributions are as follows:
\begin{itemize}
\item Motivated by the updates of the multipliers in ALM, we overcome these issues by replacing the intractable conditional expectations in the weighted augmented Lagrangian function with the multipliers. The replacement gets rid of the those difficulties without modifying the one-step reward, introducing smoothed Bellman equation and extra computations as in \cite{dai2018sbeed,vieillard2020munchausen,zhu2020borrowing}.
More crucially, it motivates a promising breakthrough to equivalently reformulate the two steps in ALM into a single constrained optimization problem. For practical consideration, we construct a quadratic penalty function for the constrained problem and employ the semi-gradient update with a target network for primal variables. Afterwards, a practical and stochastic optimization method, dubbed as SCAL, is developed. 
\item We establish theoretical analysis for the proposed algorithm.
The algorithm can be regarded as an inexact augmented Lagrangian method.
The error is generated from two parts: minimizing the augmented Lagrangian function and the update of multiplier. If these two error terms can be controlled properly, the algorithm converges to the optimal solution of the LP problems.
Then we present a specific analysis under neural tangent kernel (NTK) setting, where  the variables  behave as linear functions, so that the inner loop is similar to the optimization of convex functions. It gives a bound of the error between an outer iteration of our algorithm and an exact augmented Lagrangian update. Furthermore, we get that the residual can be arbitrarily small if the network is sufficiently wide and the parameters in algorithm is chosen properly.

\item To demonstrate the learning capability of SCAL, we take an ablation study by investigating the effects of the double-sampling solution on gradient variance and varying the number of rollout steps, and perform comprehensive numerical experiments. The results imply that our algorithm has a competitive learning potential. 

\end{itemize}



The paper is organized as follows. In section \ref{Preliminaries}, we clarify the LP formulations in RL, then introduce the ALM framework. In section \ref{Deep Parameterized ALM}, we analyze the difficulties in evaluating the parameterized augmented Lagrangian function and its gradient, and present an optimization algorithm which is equivalent to ALM but easier to calculate. We further propose a single constrained optimization problem in section \ref{The Augmented Lagrangian Method} to integrate the two steps in ALM. Some theoretical conclusions are introduced in section \ref{Theoretical Analysis}. The implementation details are explained in section \ref{implementation matters}. Comprehensive numerical experiments with several state-of-the-art RL algorithms are provided in section \ref{Numerical Experiments}. 

\section{Preliminaries}\label{Preliminaries}
We consider the standard reinforcement learning setting characterized by an infinite-horizon discounted MDP. Formally, a MDP is described by the tuple $(\mathcal{S},\mathcal{A},P,r,\rho_0, \gamma)$. We assume that $\mathcal{S}$ and $\mathcal{A}$ represent the finite state and action spaces, respectively. $P: \mathcal{S}\times \mathcal{A}\times \mathcal{S}\rightarrow \mathbf{R}$ indicates the transition dynamics by a conditional distribution over the state and action spaces, $r: \mathcal{S} \times \mathcal{A} \rightarrow \mathbf{R}$ is the bounded reward function. The initial state $s_0$ of the MDP obeys the distribution $\rho_0: \mathcal{S}\rightarrow \mathbf{R}$, and $\gamma \in (0,1)$ is the discount factor.

\subsection{LP Formulation}
Generally, we denote the total expected reward as the expectation of the cumulative discounted reward of a trajectory induced by the policy $\pi$:
\begin{equation}\label{eq:etapi-form}
\eta(\pi)=\mathbb{E}_{\pi}\left[\sum_{t=0}^\infty\gamma^tr(s_t, a_t)\right],
\end{equation}
where $s_0 \sim \rho_0, a_t \sim \pi(\cdot|s_t), s_{t+1}\sim P(\cdot|s_t,a_t)$.
The optimization problem is 
\begin{equation}\label{etapi}
\max_{\pi\in\Pi}\quad\eta(\pi),
\end{equation}
where 
\be\notag
\Pi=\left\{\pi|\sum\limits_{a\in\mathcal{A}}\pi(a|s)=1,\pi(a|s)\geq0,\forall a\in\mathcal{A}, s\in\mathcal{S}\right\}.
\ee 
The optimal state value function is the maximal expected cumulative discounted reward from state $s$, i.e., 
\be\label{value-func}
V^*(s)=\max_{\pi\in \Pi}\mathbb{E}_{\pi}\left[\sum_{t=0}^\infty\gamma^tr(s_t, a_t)\Big|s_0=s\right].
\ee 
Respectively, the optimal Bellman equation states that
\be\label{optimal-bellman}
V^*(s) =(\mathcal{T}V^*)(s) =  \max\limits_{a}\{r(s,a)+  \gamma \mathbb{E}_{s'|s,a}[ V^*(s')]\},\forall s,
\ee
which is a fixed point equation of $V^*(s)$. As the Bellman operator $\mathcal{T}$ is contractive, 
the state value function $V^*(s)$ in \eqref{value-func} is the unique solution of \eqref{optimal-bellman}. Furthermore, the optimal Bellman equation can be interpreted as solving the LP problem \cite{bertsekas1995dynamic,puterman2014markov}:
\begin{equation}\label{eq:rllp-primal}
\begin{split}
\mathrm{(P):}\quad\min_V& \sum_{s} \rho_0(s)V(s), \\
\st & V(s)\geq  r(s,a)+  \gamma \mathbb{E}_{s'|s,a}[ V(s')] , \forall s, a.
\end{split}
\end{equation}
The optimal policy $\pi^*$ can be obtained from the solution of \eqref{eq:rllp-primal} via
\be\label{Uandpi}
\pi^*(a|s) = 
\begin{cases}
1 & a = \arg\max_{a'} r(s,a')+\gamma \mathbb{E}_{s'|s,a}[ V^*(s')], \\
0 & \mathrm{o.w}.
\end{cases}
\ee
Actually, any policy that assigns positive probability only to the maximizers (i.e., actions) of the right side at state $s$ can be regarded as an optimal policy $\pi^*$. We just denote the deterministic case \eqref{Uandpi} for simplicity. In fact, a closer relation to the optimal policy $\pi^*$ comes from an equivalent form of the dual of the LP problem
\begin{equation}\label{eq:rllp-dual}
\begin{split}
\mathrm{(D):}\quad\max_{x\geq0}& \sum_{s,a} r(s,a)w(s,a)x(s,a), \\
\st& \sum_{a}w(s',a)x(s',a)- \sum_{s,a} \gamma P(s'|s,a) w(s,a)x(s,a) = \rho_0(s'), \forall s',
\end{split}
\end{equation}
where the weight function $w$ is a distribution with respect to the state and action spaces, i.e., 
\bee
\sum\limits_{s\in\mathcal{S},a\in\mathcal{A}}w(s,a)=1.
\eee
The only difference from the standard dual formulation is that the variables $x(s,a)$ carries the weight $w(s,a)$. Actually, $x$ can be viewed as the multipliers using a weighted inner product other than the standard case in Euclidean space. Moreover, the weight function $w(s,a)$ enables us to perform a sampling process from a replay buffer in large and complex environments. In practice, the explicit form of $w(s,a)$ is not required as it simply characterizes the sampling behavior.
It can be verified that \cite{cho2017deep,dai2017boosting} 
\be\label{fandpi}
\sum_{s,a}w(s,a)x^*(s,a)=\frac{1}{1-\gamma},\mathrm{\ and}\quad \pi^*(a|s) = \frac{w(s,a)x^*(s,a)}{\sum_{a'}w(s,a)x^*(s,a')},
\ee
where $x^*$ is the solution of \eqref{eq:rllp-dual}. 

Generally, the goal of RL is computing the fixed point of the optimal Bellman
equation or the maximizer of \eqref{etapi}. Alternatively, we can focus on the
LP problem \eqref{eq:rllp-primal} or \eqref{eq:rllp-dual} in an equivalent way.
However, as the transition probability $P(s'|s,a)$ is usually not accessible in real-world problems or model-free setting, the summation over $s,a$ in the constraint function of the dual form \eqref{eq:rllp-dual} limits the estimation from sampling process. Therefore, the primal LP problem \eqref{eq:rllp-primal} is considered in most cases. 

\subsection{The Classic ALM} \label{The Classic ALM}
To simplify the following derivation, we consider an equivalent primal LP formulation with equality constraints by introducing an additional slack variable. Namely, we reformulate \eqref{eq:rllp-primal} as 
\begin{equation}\label{eq:rllp-primal-eq}
\begin{split}
\min_{h\geq0,V}&\  \sum_{s} \rho_0(s)V(s), \\
\st & V(s)=  r(s,a)+  \gamma  \mathbb{E}_{s'|s,a}[ V(s')]+h(s,a), \forall s, a.
\end{split}
\end{equation}
If $a^*$ is the optimal action in state $s$, then the connection between the optimal policy and the optimal slack variable is, 
\be\label{handpi}
\begin{cases}
\pi^*(a|s)=1, h^*(s,a)=0, & a = a^*, \\
\pi^*(a|s)=0, h^*(s,a)>0, & a \neq a^*.
\end{cases}
\ee
We denote $x(s,a)$ as the multiplier of the constraints in \eqref{eq:rllp-primal-eq}. For simplicity, we introduce a $Z$-function as
\bee
Z_\mu(V,h,x,s,a) =  x(s,a)+  \mu \left(h(s,a)+r(s,a)+\gamma \mathbb{E}_{s'|s,a}[ V(s')] - V(s)\right),\label{zfunc}
\eee
and define
\bee
z_\mu(V,h,x,s,a,s') =  x(s,a)+  \mu \left(h(s,a)+r(s,a)+\gamma V(s') - V(s)\right),\label{zfunc1}
\eee
where $\mu>0$ is the penalty parameter. Since 
\bea\label{sumP=1}
\sum_{s'} P(s'|s,a) = 1,
\eea
we have
\bea\label{Z-z}
Z_\mu(V,h,x,s,a) = \mathbb{E}_{s'|s,a} [z_\mu(V,h,x,s,a,s')].
\eea
That is to say, $Z$-function is a conditional expectation of $s'$. For the LP problem \eqref{eq:rllp-primal-eq}, the weighted augmented Lagrangian function is defined as 
\begin{equation}\label{randALM}
\mathcal{L}_{\mu}(V,h,x) = \sum_{s} \rho_0(s)V(s) +\frac{1}{2\mu}\sum_{s,a}w(s,a)[Z_\mu(V,h,x,s,a)]^2.
\end{equation}

We use superscript $k$ for all variables to denote the iteration index. At the $k$-th iteration, ALM can be written as
\bea
h^{k+1},V^{k+1}&=& \arg\min_{h\geq0,V} \mathcal{L}_{\mu}(V,h,x^k), \label{eq:Vk}\\ 
x^{k+1}(s,a) &=& Z_\mu(V^{k+1},h^{k+1},x^k,s,a) ,\forall s,a. \label{eq:xk}
\eea
The general ALM performs these two steps repeatedly until some terminal conditions are satisfied. A general pseudo-code of ALM is given in Algorithm \ref{classicALMcode}.
\begin{algorithm}
		\caption{The Augmented Lagrangian Method}
		\label{classicALMcode}
		\begin{algorithmic}[1]	
		\STATE Choose initial points $V^0, h^0, x^0$, set $\mu,w$.
		\FOR{$k=1,2,...$}
			\STATE update $V^{k+1},h^{k+1}$ from the optimization \eqref{eq:Vk};	
			\STATE update the multiplier $x^{k+1}(s,a)$ for each $s,a$ with \eqref{eq:xk};
		\ENDFOR	
		\end{algorithmic}
	\end{algorithm}
Obviously, the variable size of the optimization \eqref{eq:Vk} is the same as that of the state space $\mathcal{S}$. Moreover, when the state and action spaces are extremely large or even continuous, the element-wise update in \eqref{eq:xk} is hard to compute. Thus, the general ALM is impractical due to the computational obstacle. 

\section{A Deep Parameterized ALM}\label{Deep Parameterized ALM}
In this part, we first consider a parameterized formulation to approximate the unaccessible functions in ALM and analyze its challenges in sampling and evaluation. Then we propose an approximate function which is easy to compute and present an algorithm for solving the deep parameterized ALM.
\subsection{Parameterization and Strategies for Double-Sampling}\label{A Strategy to Overcome Double-Sampling}
 We take the deep neural networks to approximate the value function, slack function and multipliers. Suppose that $V$, $h$ and $x$ are parameterized by $V_\phi$, $h_\psi$ and $x_\theta$, respectively. The non-negative property of $h_{\psi}$ is automatically guaranteed by the network. For simplicity, we directly replace $V_{\phi},h_{\psi},x_{\theta}$ in the following discussion by $\phi, \psi$ and $\theta$ without extra explanation. 
 
 Let the superscripts $\phi^k$, $\psi^k$ and $\theta^k$,  represent the parameters of the networks in the $k$-th iteration. 
Respectively, the weighted augmented Lagrangian function in \eqref{randALM} is denoted as $\mathcal{L}^k_{\mu}(\phi,\psi)$.
According to the definition, the second term in $\mathcal{L}^k_{\mu}(\phi,\psi)$ is an expectation of the squared $Z$-function over states and actions, and the $Z$-function itself is a conditional expectation over the next state. Thus, the computation of $\mathcal{L}^k_{\mu}(\phi,\psi)$ is impractical from sampling process. Furthermore, 
the gradient of $\mathcal{L}^k_{\mu}(\phi,\psi)$ with respect to $\phi$ satisfies
\be\label{grad-v}
\begin{split}
&\nabla_\phi \mathcal{L}^k_{\mu}(\phi,\psi)-\sum_{s} \rho_0(s)\nabla_{\phi} V_{\phi}(s) \\
=& \sum_{s,a}w(s,a)Z_{\mu}(\phi,\psi,\theta^{k},s,a)\left(\gamma\mathbb{E}_{s'|s,a}[\nabla_{\phi} V_{\phi}(s')]-\nabla_{\phi} V_{\phi}(s)\right)\\
=& \sum_{s,a,s'}p_w(s,a,s')Z_{\mu}(\phi,\psi,\theta^{k},s,a)\left(\gamma\nabla_{\phi} V_{\phi}(s')-\nabla_{\phi}V_{\phi}(s)\right),
\end{split}
\ee
where $p_w(s,a,s')=w(s,a)P(s'|s,a)$ can be viewed as a probability function of transition tuples. The last equation follows from \eqref{sumP=1}. The right-hand term in \eqref{grad-v} contains two expectations of the next state $s'$ for each pair $(s,a)$: one is the gradient of the value function at $s'$ and the other comes from the definition of the $Z$-function. Thus, to estimate the gradient $\nabla_{\phi}\mathcal{L}_{\mu}^k(\phi,\psi)$, independent samples of the next state $s'$ are required for evaluating $\nabla_{\phi} V_{\phi}(s')$ and approximating $Z_{\mu}(\phi,\psi,\theta^{k},s,a)$ separately, while only one $s'$ is observed from the pair $(s,a)$ in sampling case. This arises the well-known double-sampling obstacle.


To overcome the challenges of estimating $\mathcal{L}^k_{\mu}(\phi,\psi)$ and its gradient, we consider approximating the intractable function $Z_{\mu}(\phi,\psi,\theta^{k},s,a)$ with an additional function $y_{\tau}(s,a)$. The weight of the network $\tau$ is trained to approximate the $Z$-function with a weighted least square regression, i.e.,
\be\label{para-x}
\min_\tau\sum\limits_{s,a}w(s,a)\left(Z_{\mu}(\phi,\psi,\theta^{k},s,a)-y_{\tau}(s,a)\right)^2.
\ee
By adding a term unrelated to $\tau$ as 
\bee
\sum_{s,a,s'} p_w(s,a,s')z_\mu(\phi,\psi,\theta^k,s,a,s')^2 - \sum_{s,a} w(s,a)Z_\mu(\phi,\psi,\theta^k,s,a)^2
\eee
to the objective function of \eqref{para-x} and using the relationships in \eqref{sumP=1} and \eqref{Z-z}, we can equivalently transform the optimization problem \eqref{para-x} into 
\be\label{eq:wls2}
\min_{\tau}  \sum_{s,a,s'}p_w(s,a,s')\left(z_\mu(\phi,\psi,\theta^k,s,a,s') -  y_{\tau}(s,a)  \right)^2,
\ee
which enables stochastic gradient optimization by taking samples on the transition tuples $(s, a, r, s')$. Consequently, an approximation of $\mathcal{L}^k_{\mu}(\phi,\psi)$ is defined as
\begin{equation}\label{Lwithx}
\begin{split}
\tilde{\mathcal{L}}^k_{\mu}(\phi,\psi,\tau) = &\sum_{s} \rho_0(s)V_{\phi}(s) +\frac{1}{\mu}\sum_{s,a}w(s,a)y_{\tau}(s,a)Z_\mu(\phi,\psi,\theta^{k},s,a)\\
=&\sum_{s} \rho_0(s)V_{\phi}(s) +\frac{1}{\mu}\sum_{s,a,s'}p_w(s,a,s')y_{\tau}(s,a)z_\mu(\phi,\psi,\theta^{k},s,a,s'),
\end{split}
\end{equation}
which is a summation of ordinary expectations over initial states and transition tuples. Therefore, its value and gradient are relatively easy to compute. Suppose that $y_{\tau}(s,a)$ is the solution of \eqref{para-x} with respect to $\phi,\psi$, then the gradients $\nabla_{\phi}\tilde{\mathcal{L}}^k_{\mu}(\phi,\psi,\tau)$ and $\nabla_{\psi}\tilde{\mathcal{L}}^k_{\mu}(\phi,\psi,\tau)$ are exactly the same as those of $\mathcal{L}^k_{\mu}(\phi,\psi)$, respectively. Furthermore, if $(\phi^{k+1},\psi^{k+1})$ is the minimizer of $\mathcal{L}^k_{\mu}(\phi,\psi)$, then the corresponding function $y_{\tau^{k+1}}(s,a)$ is identical to the multiplier $x_{\theta^{k+1}}(s,a)$. In other words, it implies that the optimization problem \eqref{para-x} is equivalent to the multiplier update in \eqref{eq:xk}. These properties provide a particular strategy on solving the deep parameterized ALM.

\subsection{A Deep Parameterized ALM}
Based on the proposed strategy for handling double-sampling obstacle in $\mathcal{L}^k_{\mu}(\phi,\psi)$, we can directly take the multiplier $x_{\theta}$ to substitute the additional function $y_{\tau}$ in \eqref{eq:wls2}, as well as in \eqref{Lwithx}. Generally, by denoting
\begin{equation*}\label{Lwithx-final}
\begin{split}
\tilde{\mathcal{L}}^k_{\mu}(\phi,\psi,\theta) =\sum_{s} \rho_0(s)V_{\phi}(s) +\frac{1}{\mu}\sum_{s,a,s'}p_w(s,a,s')x_{\theta}(s,a)z_\mu(\phi,\psi,\theta^{k},s,a,s'),
\end{split}
\end{equation*}
the deep parameterized ALM performs
\bea
\phi^{k+1} ,\psi^{k+1}&=& \arg\min_{\phi,\psi}\tilde{\mathcal{L}}^k_{\mu}(\phi,\psi,\theta^k), \notag\\ 
\theta^{k+1} &=& \arg\min_{\theta}\sum_{s,a,s'}p_w(s,a,s')[z_\mu(\phi^{k+1},\psi^{k+1},\theta^k,s,a,s')-x_{\theta}(s,a)]^2. \notag
\eea
The multiplier update is similar as the derivation from \eqref{para-x} to \eqref{eq:wls2}. Practically, we can approximate these two minimization problems with single or multiple stochastic gradient steps. Namely, at $k$-th iteration, starting from $\phi^k_0 = \phi^k$, $\psi^k_0 = \psi^k$ and $\theta^k_0 =\theta^k$, the following three operations are performed at each inner step $t$:
\bea
\phi^k_{t+1} &=& \phi^k_t - \beta_{\phi}\cdot\hat{\nabla}_{\phi}\tilde{\mathcal{L}}^k_{\mu}(\phi^{k}_t,\psi^k_t, \theta^k_{t+1}), \label{update-V}\\
\psi^k_{t+1} &=& \psi^k_t - \beta_{\psi}\cdot\hat{\nabla}_{\psi}\tilde{\mathcal{L}}^k_{\mu}(\phi^{k}_{t+1},\psi^k_{t},\theta^k_{t+1}), \label{update-h}\\
\theta^k_{t+1} &=&\theta^k_t - \beta_{\theta}\cdot\left(x_{\theta}(s,a)-z_\mu(\phi^{k}_{t},\psi^{k}_{t},\theta^k,s,a,s') \right)\nabla_{\theta} x_{\theta}(s,a),\label{update-x}
\eea
where $\beta_{\phi},\beta_{\psi},\beta_{\theta}$ are the step sizes. The inner iteration terminates until some stop criterion holds, and the last iterators  $\phi^k_t, \psi^k_t$ and $\theta^k_{t}$ are assigned to $\phi^{k+1}, \psi^{k+1}$ and $\theta^{k+1}$, respectively. 

\section{A Stochastic Composite Augmented Lagrangian Algorithm}\label{The Augmented Lagrangian Method}

The multiplier $x_{\theta}(s,a)$ may not fit the $Z$-function well in the deep parameterized ALM since the number of samples is limited and the step \eqref{update-x} is inexact. We observe that the performance is not stable in practice and the approximate function $\tilde{\mathcal{L}}^k_{\mu}(\phi,\psi,\theta)$ is often not possible to generate desirable update direction for the primal variable $\phi$ and slack variable $\psi$. Therefore, the iteration \eqref{update-V}-\eqref{update-x} is highly sensitive to the hyper-parameters and usually results in an unstable performance.

\subsection{A Quadratic Penalty Optimization}
Based on the connections between $\tilde{\mathcal{L}}^k_{\mu}(\phi,\psi,\theta)$ and $\mathcal{L}^k_{\mu}(\phi,\psi)$, we can incorporate the two steps in parameterized ALM with a single constrained optimization. We consider the following constrained problem
\be\begin{split}\label{constrained-alf}
\min_{\phi,\psi,\theta}\ \tilde{\mathcal{L}}^k_{\mu}(\phi,\psi,\theta),\ \mathrm{s.t.}\ x_{\theta}(s,a)=Z_\mu(\phi,\psi,\theta^{k},s,a),\forall s,a.
\end{split}
\ee
It is worth noting that, the coefficient in the second term of $\tilde{\mathcal{L}}^k_{\mu}(\phi,\psi,\theta)$ is $\frac{1}{\mu}$, which is different from that of the augmented Lagrangian function \eqref{randALM}, i.e., $\frac{1}{2\mu}$. However, $\tilde{\mathcal{L}}^k_{\mu}(\phi,\psi,\theta)$ is equivalent to the augmented Lagrangian function of the primal LP problem \eqref{eq:rllp-primal} when its objective function is multiplied with two. Actually, multiplying any positive scalar for the LP problem does not change the solution of \eqref{eq:rllp-primal}. That is to say, the constrained optimization \eqref{constrained-alf} is an integration of the two steps in ALM. We can simultaneously accomplish the updates of the primal, slack and multiplier variables by solving a single constrained optimization problem.

In order to approximately solve \eqref{constrained-alf}, we consider penalizing the constraints into the objective function to form a quadratic penalty optimization, i.e.,
\be\label{comp-func}
\begin{split}
\min_{\phi,\psi,\theta}\ \tilde{\mathcal{L}}^k_{\beta,\mu}(\phi,\psi,\theta):=&\tilde{\mathcal{L}}^k_{\mu}(\phi,\psi,\theta)\\%
&+\frac{\beta}{2}\sum_{s,a} w(s,a)\left[x_{\theta}(s,a) - \left(Z_\mu(\phi,\psi,\theta^{k},s,a)\right)\right]^2,
\end{split}
\ee
where $\beta>0$. 
We employ the semi-gradient for $\phi$, specifically, the term $r(s,a)+ \gamma\sum_{s'}P(s'|s,a)V_{\phi}(s')$ in $Z$-function is treated as a target. This is a bootstrapping way originated from temporal-difference learning and is commonly used in many RL algorithms, such as in \cite{mnih2015human,wang2017randomized}. Therefore, we introduce a target value network $V_{\phi^\mathrm{targ}}$ as the usage in DQN. Additionally, we also replace $x_{\theta^k}$ in \eqref{comp-func} with a target multiplier network $x_{\theta^{\mathrm{targ}}}$ which shares the same structure as that of $x_{\theta}$. We denote
\bea
z^{\mathrm{targ}}_\mu(\phi,\psi,s,a,s') &=& x_{\theta^{\mathrm{targ}}}(s,a)+  \mu \left(h_{\psi}(s,a)+r(s,a)+ \gamma V_{\phi^{\mathrm{targ}}}(s') - V_{\phi}(s)\right),\notag\\
Z^{\mathrm{targ}}_\mu(\phi,\psi,s,a) &= &\mathbb{E}_{s'|s,a}[z^{\mathrm{targ}}_\mu(\phi,\psi,s,a,s')].\notag
\eea
Then the objective function in \eqref{comp-func} is reformulated as
\begin{equation}\label{comp-func-targ}
\begin{split}
&\hat{\mathcal{L}}^k_{\beta,\mu}(\phi,\psi,\theta)\\
&:=\sum_{s} \rho_0(s)V_{\phi}(s) + \frac{1}{\mu}\sum_{s,a,s'}p_w(s,a,s')x_{\theta}(s,a)z^{\mathrm{targ}}_\mu(\phi,\psi,s,a,s') \\
&\ \ \ +\frac{\beta}{2}\sum_{s,a} w(s,a)\left[x_{\theta}(s,a) - \left(Z^{\mathrm{targ}}_\mu(\phi,\psi,s,a)\right)\right]^2.
\end{split}\end{equation}
By adding a term unrelated to $\phi,\theta$ and $\psi$ as 
\bea\notag
&&\sum_{s,a,s'} p_w(s,a,s')[x_{\theta^{\mathrm{targ}}}(s,a)+\mu\left( r(s,a)+\gamma V_{\phi^{\mathrm{targ}}}(s')\right)]^2 \\
&&- \sum_{s,a} w(s,a)\left(x_{\theta^\mathrm{targ}}(s,a)+\mu\left[ r(s,a)+\gamma\sum_{s'}P(s'|s,a) V_{\phi^{\mathrm{targ}}}(s')\right]\right)^2\notag
\eea
to the last square function of \eqref{comp-func-targ}, we can rearrange it as
\begin{equation*}
\begin{split}
&\hat{\mathcal{L}}^{k}_{\beta,\mu}(\phi,\psi,\theta)\\
&=\sum_{s} \rho_0(s)V_{\phi}(s) +  \frac{1}{\mu}\sum_{s,a,s'}p_w(s,a,s')x_{\theta}(s,a)z^{\mathrm{targ}}_\mu(\phi,\psi,s,a,s')\\
&\ \ \ +\frac{\beta}{2}\sum_{s,a,s'} p_w(s,a,s')\left(x_{\theta}(s,a) - z^{\mathrm{targ}}_\mu(\phi,\psi,s,a,s')\right)^2,
\end{split}\end{equation*}
which can be easily evaluated with states and transitions sampled from $\rho_0$ and $p_w(s,a,s')$, respectively. 
In the $k$-th iteration, we need to solve the following optimization problem
\bea
\phi^{k+1},\psi^{k+1},\theta^{k+1}  = \arg\min_{\phi,\psi,\theta}\ \hat{\mathcal{L}}^k_{\beta,\mu}(\phi,\psi,\theta)\label{compositional-obj} 
\eea
with appropriate hyper-parameters. 
\subsection{Stochastic Optimization}
We next consider the sampling process and give a stochastic algorithm. To be clear in advance, as the explicit form of the transition probability $P(s'|s,a)$ is usually unknown in practice, our derivation is constructed in the model-free setting. Therefore, to promote the data efficiency in an online RL setting, an experience replay buffer $\mathcal{D}$ is adopted to store the initial states $\{s_0\}$ and the transition tuples $\{(s_i,a_i,r_i,s'_{i})\}$ encountered in each episode. At each step, we randomly get a batch of initial state samples $\{s_{0,i}\}_{i=1}^b$ and transition tuples $\{(s_i,a_i,r_i,s'_{i})\}_{i=1}^b$ ($b$ is the batch size) from the replay buffer $\mathcal{D}$, in which the samples are assumed to be sampled from the probability $\rho_0(s)$ and $p_w(s,a,s')$, respectively. We denote
\begin{equation}\label{grad_tems}
\begin{split}
e_i =& \beta\left(z^{\mathrm{targ}}_\mu(\phi^k,\psi^k,s_i,a_i,s'_i)-x_{\theta^{k}}(s_i,a_i)\right),\\
g_i =& \nabla_{\phi}V_{\phi^k}(s_{0,i})-\left(x_{\theta^{k}}(s_i,a_i)+\mu e_i\right) \nabla_{\phi}V_{\phi^k}(s_i),\\
q_i=&(x_{\theta^k}(s_i,a_i)+\mu e_i)\nabla_{\psi} h_{\psi^k}(s_i,a_i),\\
m_i=&\frac{1}{\mu}\left(z^{\mathrm{targ}}_\mu(\phi^k,\psi^k,s_i,a_i,s'_i)-\mu e_i\right)\nabla_{\theta} x_{\theta^k}(s_i,a_i),
\end{split}\end{equation}
and update the parameters as
\begin{equation}\label{comp-phi-theta-psi}
\begin{split}
\phi^{k+1} =&\phi^k - \frac{\alpha_{\phi}}{b}\sum_{i=1}^bg_i,\\
\psi^{k+1} =&\psi^k-\frac{\alpha_{\psi}}{b}\sum_{i=1}^bq_i,\\
\theta^{k+1} =&\theta^k-\frac{\alpha_{\theta}}{b}\sum_{i=1}^bm_i,\\
\end{split}\end{equation}
where $\alpha_{\phi},\alpha_{\theta},\alpha_{\psi} >0$ are step size. We assign the current parameters $\theta^k$ and $\phi^k$ to the targets $\theta^{\mathrm{targ}}$ and $\phi^{\mathrm{targ}}$ for every $T$ iterations. The update rules in \eqref{comp-phi-theta-psi} is a type of Jacobi iteration. Alternatively, we can apply Gauss-Seidel iteration. However, the gradient estimation in deep neural network composes the forward and backward propagation. Thus, in consideration of the computational cost, Jacobi iteration is more practical.

In conclusion, we instantiate our method named SCAL in Algorithm \ref{DSALMcode}, combined with an experience replay $\mathcal{D}$ with a behavior policy (Line 4-5). Intuitively, the multiplier $x_{\theta^k}$ is approaching $x_{\theta^*}$ as the algorithm iterates. Therefore, we can exploit the policy obtained in previous iteration, i.e., $x_{\theta^{k}}$, as the behavior policy $\pi^{b}$. Lines 6-7 correspond to the stochastic gradient descents in \eqref{comp-phi-theta-psi}.
\begin{algorithm}
		\caption{A Stochastic Composite ALM (SCAL)}
		\label{DSALMcode}
		\begin{algorithmic}[1]	
		\STATE Initialize the points $ \phi^0,\psi^0,\theta^0$. Set step size $\alpha_{\phi},\alpha_{\psi},\alpha_{\theta}\in(0,1)$.
		\STATE Set $\phi^{\mathrm{targ}}=\phi^0 $ and $\theta^{\mathrm{targ}}=\theta^0 $;
		\FOR{$k = 0,...,K$}
		\STATE add a new transition $(s,a,r,s')$ into $\mathcal{D}$ by executing the behavior policy $\pi_{b}$.

		\STATE sample $\{(s_i,a_i,r_i,s'_i)\}_{i=1}^b$ and $\{s_{0,i}\}_{i=1}^b$from $\mathcal{D}$;
			
			\STATE compute $\{g_i\}_{i=1}^b,\{q_i\}_{i=1}^b$ and $\{m_i\}_{i=1}^b$ using \eqref{grad_tems};
			
			\STATE  update $\phi^{k+1},\psi^{k+1}$ and $\theta^{k+1}$ using \eqref{comp-phi-theta-psi}  ;

	\IF{k \% T == 0}
	\STATE Set $\phi^{\mathrm{targ}}=\phi^{k+1}$ and $\theta^{\mathrm{targ}}=\theta^{k+1}$;
	\ENDIF	
		
\ENDFOR	
		\end{algorithmic}
	\end{algorithm}

\section{Theoretical Analysis}\label{Theoretical Analysis}

\subsection{General convergence}
In this section, we analyze the convergence of the iteration \eqref{constrained-alf}. It is worth recalling the fact that the update \eqref{constrained-alf} is equivalent to an approximated augmented Lagrangian update \eqref{eq:Vk}-\eqref{eq:xk} upto a scaling. Multiplying the objective function in \eqref{eq:rllp-primal} by a constant does not change the optimal solution of \eqref{eq:rllp-primal}. 
We consider the difference between exact augmented Lagrangian update and the iteration \eqref{constrained-alf}. 
Define the error terms
\beaa
\epsilon_L^k &=& \mathcal{L}_{\mu}(V_{\phi^{k+1}},h_{\psi^{k+1}}, x_{\theta^{k}}) -\inf_{h \geq 0, V} \mathcal{L}_{\mu}(V,h,x_{\theta^k}), \\ 
\epsilon_x^k &=& \sum_{s,a} w(s,a)\left( x_{\theta^{k+1}}(s,a) - Z_\mu(V_{\phi^{k+1}},h_{\psi^{k+1}}, x_{\theta^k}, s,a) \right)^2.
\eeaa
Notice that the error terms $\epsilon_L^k$ and $\epsilon_x^k$ contain the optimization error and parameterization error. The first part is generated from the optimization algorithm and the second part is due to the parametrization way to represent variables. We show that if the error terms are bounded well, the update \eqref{constrained-alf} converges to the optimal solution of \eqref{eq:rllp-dual}.

The following lemma gives a connection between the proximal point method and augmented Lagrangian method. It shows that $Z_\mu(V,h,x,s,a)$ is a solution of an optimization problem.
\begin{lemma}\label{lem:lp}
    If $V^*,h^* = \arg\min_{h \geq 0,V}  \mathcal{L}_{\mu}(V,h, x')$, then the term $Z_\mu(V^*,h^*,x', s,a)$ is an optimal solution of the following problem
\begin{equation}\label{eq:prox-lp}
\begin{split}
\max_x& \sum_{s,a} r(s,a)w(s,a)x(s,a) - \frac{1}{2\mu}\sum_{s,a} w(s,a)(x(s,a) - x'(s,a))^2  \\
\st&  \sum_{s,a} (\delta_{s'}(s) - \gamma P(s'|s,a)) w(s,a)x(s,a) = \rho_0(s'), \forall s' \\
&x(s,a) \geq 0, \forall s,a.
\end{split}
\end{equation}
The optimal value of problem \eqref{eq:prox-lp} is equal to 
\be\label{eq:Fmu}
\mathcal{F}_{\mu}(V^*,h^*,x') := \mathcal{L}_{\mu}(V^*,h^*, x') - \frac{1}{2\mu}\sum_{s,a} w(s,a)x'(s,a)^2.
\ee

\end{lemma}
\begin{proof}
The KKT condition of the problem \eqref{eq:prox-lp} can be written as  
\bea
\begin{split}
&x(s,a) = x'(s,a) + \mu \left(h(s,a) + r(s,a)+\sum_{s'} \gamma P(s'|s,a) V(s') - V(s)\right), \forall s,a,\\
&\sum_{s,a} (\delta_{s'}(s) - \gamma P(s'|s,a)) w(s,a)x(s,a) = \rho_0(s'), \forall s', \label{eq:kkt} \\
&x(s,a) \geq 0, h(s,a) \geq 0, \forall s,a,
\end{split}
\eea
where $V$ and $h$ are the dual variables. Then, $x^*$ is an optimal solution if and only if there exists variables $V$ and $h$ such that $x^*, V$ and $h$ satisfy the condition \eqref{eq:kkt}. Suppose that $V^*,h^* = \arg\min_{h\geq 0,V}  \mathcal{L}_{\mu}(V,h, x')$. We have the optimality condition
\be\label{eq:kkt-lg}
    \begin{split}
    &\rho_0(s') - \sum_{s,a}w(s,a)Z_\mu(V^*,h^*,x',s,a)(\delta_{s'}(s) - \gamma P(s'|s,a)) = 0, \forall s' \\
    & \mu h(x,a) = \left[- x'(s,a) - \mu\left(r(s,a)+\sum_{s'} \gamma P(s'|s,a) V(s') - V(s)\right)\right]_+.
\end{split}
\ee
Let 
\bee
    \begin{split}
    x(s,a) &= Z_\mu(V^*,h^*,x', s,a)\\ 
    &= \left[x'(s,a) + \mu \left(r(s,a)+\sum_{s'} \gamma P(s'|s,a) V(s') - V(s)\right)\right]_+ \geq 0.
    \end{split}
\eee
Then the KKT condition \eqref{eq:kkt} holds, i.e., $x(s,a) = Z_\mu(V^*,h^*,x',s,a)$ is an optimal solution of \eqref{eq:prox-lp}.

Let $x^* = Z_\mu(V^*,h^*,x',s,a)$, then we obtain
\bee
\begin{split}
&\mathcal{F}_{\mu}(V^*,h^*,x') \\
= & \sum_s \rho_0(s) V(s) + \frac{1}{2\mu}\sum_{s,a}w(s,a) x^*(s,a)^2 - \frac{1}{2\mu}\sum_{s,a} w(s,a)x'(s,a)^2 \\
 = & \sum_s \rho_0(s) V(s) + \frac{1}{\mu}\sum_{s,a}w(s,a) x^*(s,a)(x^*(s,a)-x'(s,a)) \\
& - \frac{1}{2\mu}\sum_{s,a} w(s,a)(x^*(s,a)-x'(s,a))^2 \\
= & \sum_{s,a} r(s,a)w(s,a)x^*(s,a) - \frac{1}{2\mu}\sum_{s,a} w(s,a)(x^*(s,a)-x'(s,a))^2, 
\end{split}
\eee
where the last inequality is due to the first equation in \eqref{eq:kkt-lg} and the definition of $Z_\mu(V^*,h^*,x',s,a)$, which completes the proof.
\end{proof}

The next lemma shows the noise between one exact augmented Lagrangian update and 
the corresponding approximated update can be bounded by the two error terms $\epsilon_x^k$ and $\epsilon_L^k$.
\begin{lemma}\label{lem:ppt}
    Suppose that $\theta^k$ is generated by the iteration \eqref{constrained-alf}. Define $\hat{V}^{k},\hat{h}^k = \arg\min_{h\geq 0,V}  \mathcal{L}_{\mu}(V,h, x_{\theta^k})$ and let 
    \bee
        \hat{x}^{k+1}(s,a) = Z_\mu( \hat{V}^{k}, \hat{h}^k, x_{\theta^k},s,a).
    \eee
    Then we have
\bee
\sum_{s,a} w(s,a) (x_{\theta^{k+1}}(s,a) - \hat{x}^{k+1}(s,a) )^2 \leq  2\epsilon_x^k + 4\mu\epsilon_L^k.
\eee

\end{lemma}
\begin{proof}
    Let $\bar{x}^{k+1} = Z_\mu(V_{\phi^{k+1}},h_{\psi^{k+1}}, x_{\theta^k},s,a)$
    and $\mathcal{F}_{\mu}(V,h,x)$ is defined in \eqref{eq:Fmu}.
    It is easy to show that $\mathcal{F}_{\mu}(V,h,x)$ be a concave function with respect to $x$ and 
    \bee
        [\nabla_x \mathcal{F}_{\mu}(V,h,x)](s,a) = \frac{1}{\mu}\omega(s,a)\left(Z_\mu(V,h,x,s,a) - x(s,a)\right). 
    \eee
    It follows that for any $\kappa(s,a)$,  
    \bea 
    && \mathcal{F}_{\mu} (V_{\phi^{k+1}},h_{\psi^{k+1}}, x_{\theta^k}) \\ 
    &&+ \frac{1}{\mu}\sum_{s,a}w(s,a)(\bar{x}^{k+1}(s,a)-  x_{\theta^k}(s,a))(\kappa(s,a) -  x_{\theta^k}(s,a)) \label{eq:ieq-phi} \nonumber\\
    &\geq& \mathcal{F}_{\mu} (V_{\phi^{k+1}},h_{\psi^{k+1}}, \kappa) \geq \inf_{h\geq 0,V} \mathcal{F}_{\mu} (V,h, \kappa). \nonumber 
    \eea
    According to Lemma \ref{lem:lp}, the value $\inf_{h\geq 0,V} \mathcal{F}_{\mu} (V,h, \kappa)$
    is equal to the optimal value of the following problem
    \begin{equation}\label{eq:prox-lp-1}
\begin{split}
\max_x& \sum_{s,a} r(s,a)w(s,a)x(s,a) - \frac{1}{2\mu}\sum_{s,a} w(s,a)(x(s,a) - \kappa(s,a))^2  \\
\st&  \sum_{s,a} (\delta_{s'}(s) - \gamma P(s'|s,a)) w(s,a)x(s,a) = \rho_0(s'), \forall s' \\
&x(s,a) \geq 0, \forall s,a.
\end{split}
\end{equation}
Since $\hat{x}^{k+1}(s,a)$ satisfies the constraints of problem \eqref{eq:prox-lp-1}, it implies
\be\label{eq:ieq-phi-2}
\begin{split}
\inf_{h\geq 0,V} \mathcal{F}_{\mu} (V,h, \kappa) \geq& \sum_{s,a} r(s,a)w(s,a)\hat{x}^{k+1}(s,a) \\
&- \frac{1}{2\mu}\sum_{s,a} w(s,a)(\hat{x}^{k+1}(s,a) - \kappa(s,a))^2.
\end{split}
\ee
A similar argument yields
\be\label{eq:ieq-phi-3}
\begin{split}
\inf_{h\geq 0,V} \mathcal{F}_{\mu} (V,h, x_{\theta^k}) =& \sum_{s,a} r(s,a)w(s,a)\hat{x}^{k+1}(s,a) \\
&- \frac{1}{2\mu}\sum_{s,a} w(s,a)(\hat{x}^{k+1}(s,a) - x_{\theta^k}(s,a))^2.
\end{split}
\ee
Combining \eqref{eq:ieq-phi}, \eqref{eq:ieq-phi-2} and \eqref{eq:ieq-phi-3}, we obtain
\beaa
\epsilon_L^k &=& \mathcal{L}_{\mu}(V_{\phi^{k+1}},h_{\psi^{k+1}}, x_{\theta^{k}}) -\inf_{h\geq 0,V} \mathcal{L}_{\mu}(V, h, x_{\theta^k}) \\
&=& \mathcal{F}_{\mu} (V_{\phi^{k+1}}, h_{\psi^{k+1}}, x_{\theta^k}) - \inf_{h\geq 0,V} \mathcal{F}_{\mu}(V, h, x_{\theta^k}) \\
&\geq& - \frac{1}{\mu}\sum_{s,a}w(s,a)(\bar{x}^{k+1}(s,a)-  x_{\theta^k}(s,a))(\kappa(s,a) -  x_{\theta^k}(s,a)) \\ 
&& - \frac{1}{2\mu}\sum_{s,a} w(s,a)(\hat{x}^{k+1}(s,a) - \kappa(s,a))^2 \\
&& + \frac{1}{2\mu}\sum_{s,a} w(s,a)(\hat{x}^{k+1}(s,a) - x_{\theta^k}(s,a))^2 \\
&=&\frac{1}{\mu} \sum_{s,a} w(s,a)(\kappa(s,a)-x_{\theta^k}(s,a))(\hat{x}^{k+1}(s,a) - \bar{x}^{k+1}(s,a))  \\
&&- \frac{1}{2\mu}\sum_{s,a}w(s,a)(\kappa(s,a) - x_{\theta^k}(s,a)).
\eeaa
Since $\kappa$ is arbitrary, taking $\kappa = \hat{x}^{k+1} - \bar{x}^{k+1} + x_{\theta^k} $ gives
\bee
    \epsilon_L^k \geq  \frac{1}{2\mu}\sum_{s,a} w(s,a) (\hat{x}^{k+1}(s,a) - \bar{x}^{k+1}(s,a))^2. 
\eee
It follows that
\beaa
    &&\sum_{s,a} w(s,a) (x_{\theta^{k+1}}(s,a) - \hat{x}^{k+1}(s,a) )^2 \\
    &\leq& 2\sum_{s,a} w(s,a) (x_{\theta^{k+1}}(s,a) - \bar{x}^{k+1}(s,a) )^2 \\ 
    && + 2\sum_{s,a} w(s,a) (\bar{x}^{k+1}(s,a) - \hat{x}^{k+1}(s,a) )^2 \\
    &\leq& 2\epsilon_x^k + 4\mu\epsilon_L^k,
\eeaa
which completes the proof.
\end{proof}

We next give the main convergence theorem.
\begin{theorem}
If the error terms $\epsilon^k_L$ and $\epsilon^k_x$ satisfy that $\sum_k \epsilon^k_L < \infty$ and  $\sum_k \epsilon^k_x < \infty$, then $x_{\theta^k}$ generated by iteration \eqref{constrained-alf} converges to the optimal solution of \eqref{eq:rllp-primal}.
\end{theorem}
\begin{proof}
We first claim that iteration \eqref{constrained-alf} is equalvalent to one iteration in the augmented Lagrangian method under parameterization, i.e.,
\be\label{eq:Vk-para}
\begin{split}
\phi^{k+1} ,\psi^{k+1}&= \arg\min_{\phi,\psi} \frac{1}{2}\sum_{s} \rho_0(s)V_\phi(s) +\frac{1}{2\mu}\sum_{s,a}w(s,a)[Z_\mu(V_\phi,h_\psi,x_{\theta^k},s,a)]^2, \\ 
\theta^{k+1} &= \arg\min_\theta \sum\limits_{s,a}w(s,a)\left(Z_{\mu}(\phi,\psi,\theta^{k},s,a)-x_{\theta}(s,a)\right)^2. 
\end{split}
\ee
The constraint in \eqref{constrained-alf} is the second update in \eqref{eq:Vk-para}. If we replace $x$ by the constraints in $\tilde{\mathcal{L}}^k_{\mu}(\phi,\psi,\theta)$, the optimization problem is the first step in update \eqref{eq:Vk-para}.
According to Lemmas \ref{lem:lp} and \ref{lem:ppt}, the iterations \eqref{eq:Vk-para} are equal to the proximal point method with noise $2\epsilon_x^k + 4\mu\epsilon_L^k$. It follows from \cite[Theorem 1]{rockafellar1976augmented} that $x_{\theta^k}$ converges to the optimal solution of \eqref{eq:rllp-primal}. 
\end{proof}
%
%
%
%
%

\subsection{Convergence analysis under neural tangent kernel}
In this subsection, we analyze the convergence of the scheme \eqref{comp-func} under neural tangent kernel setting. 
For convenience, we consider the discrete case (state and action space is discrete). The conclusions for the continuous case can be proven similarly under proper assumptions.
Neural tangent kernel is a basic tool to understand the approximation of neural network. It has been used to analyze the convergence of RL algorithms under neural network parametrization. 

We next parametrize the variables by the following two-layer neural networks:
\bea
V_{\phi} (s) &=& \frac{1}{\sqrt{m_V}} \sum_{i=1}^{m_V} \phi_{1,i} \cdot \sigma(\phi_{2,i}^Ts), \\ 
h_{\psi} (s,a) &=& \frac{1}{\sqrt{m_h}} \sum_{i=1}^{m_h} \psi_{1,i} \cdot \sigma(\psi_{2,i}^T(s,a)), \\
x_{\theta} (s,a) &=& \frac{1}{\sqrt{m_x}} \sum_{i=1}^{m_x} \theta_{1,i} \cdot \sigma(\theta_{2,i}^T(s,a)),
\eea
where $m_V, m_h, m_x$ are integers, $\sigma =  \max(0, x)$ is the ReLU function. The parameters are randomly initialized by either the uniform distribution or normal distribution, i.e.,
\be\label{eq:dist}
\begin{split}
&\phi_{1,i}, \theta_{1,i} \sim \mathcal{U}(\{1,-1\}), \quad \psi_{1,i} = 1, \\
&\phi_{2,i} \sim \mathcal{N}(0, I_{d_{s}}/d_s),\quad \psi_{2,i},\theta_{1,i} \sim \mathcal{N}(0, I_{d_{sa}}/d_{sa}), 
\end{split}
\ee
where $d_{s}$ is the dimension of $s$ and $d_{sa}$ is the dimension of $(s,a)$. 
In particular, $\phi_{1,i}, \psi_{1,i}, \theta_{1,i}$ are fixed and we only optimize the parameter $\phi_{2,i}, \psi_{2,i}, \theta_{2,i}$.
Let $\phi = (\phi_{2,i})_{i=1}^{m_V}$, $\psi = (\psi_{2,i})_{i=1}^{m_V}$ and $\theta = (\theta_{2,i})_{i=1}^{m_V}$.
The corresponding random initialization parameters are denoted by $\phi^0, \psi^0, \theta^0$. The expectation under the distribution \eqref{eq:dist} is denoted by $E_0$.

Define the following function classes 
\beaa
\mathcal{F}^V_{R_V, m_V} &=& \{\frac{1}{\sqrt{m_V}} \sum_{i=1}^{m_V} \phi_{1,i} \cdot 1_{\{(\phi_{2,i}^0)^Ts \geq 0 \}} \phi_{2,i}^Ts : \ \|\phi-\phi^0\| \leq R_V\},\\
\mathcal{F}^h_{R_h, m_h} &=& \{\frac{1}{\sqrt{m_h}} \sum_{i=1}^{m_h} \psi_{1,i} \cdot 1_{\{(\psi_{2,i}^0)^T(s,a) \geq 0 \}} \psi_{2,i}^T(s,a): \ \|\psi-\psi^0\| \leq R_h\}, \\
\mathcal{F}^x_{R_x, m_x} &=& \{\frac{1}{\sqrt{m_h}} \sum_{i=1}^{m_h} \theta_{1,i} \cdot 1_{\{(\theta_{2,i}^0)^T(s,a) \geq 0 \}} \theta_{2,i}^T(s,a): \ \|\theta-\theta^0\| \leq R_x\},
\eeaa
and let 
\bee
\Xi = \{\zeta: \|\phi-\phi^0\| \leq R_V, \|\psi-\psi^0\| \leq R_h \text{ and } \|\theta-\theta^0\| \leq R_x\}.
\eee
For simplicity, we assume  that $\|(s,a)\| \leq 1$, which follows that $\|s\| \leq 1$. On $\Xi$, we have
\be\label{eq:bv}
\begin{split}
&|V_\phi(s) - V_{\phi^0}(s)| \leq \frac{1}{\sqrt{m_V}} \sum_{i=1}^{m_V} |\phi_{1,i}| \cdot |\sigma(\phi_{2,i}^Ts) - \sigma((\phi_{2,i}^0)^Ts)| \\
 \leq& \frac{1}{\sqrt{m_V}} \sum_{i=1}^{m_V} |\phi_{1,i}| \cdot |\phi_{2,i}^Ts - (\phi_{2,i}^0)^Ts| \leq \frac{1}{\sqrt{m_V}} \sum_{i=1}^{m_V} \|\phi_{2,i} - (\phi_{2,i}^0)\| \leq R_V^{1/2}.
 \end{split}
\ee
Similarly, we have $|h_\psi(s) - h_{\psi^0}(s)| \leq R_h^{1/2}$ and $|x_\theta(s) - x_{\theta^0}(s)| \leq R_x^{1/2}$.
We can also obtain that
\beaa
\|\nabla_\phi V_\phi(s) \|^2 = \frac{1}{m_V} \sum_{i=1}^{m_V} 1_{\{(\phi_{2,i})^Ts \geq 0 \}} \cdot \|s\|^2 \leq 1.
\eeaa
We also yields $\|\nabla_\psi h_\psi(s,a) \| \leq  1$ and $\|\nabla_\theta x_\theta(s,a) \| \leq  1$.
Consider that $(V,h,x)$ is in the Hilbert space $\mathcal{H}$ with inner product
\[
\iprod{(V,h,x)}{(V',h',x')}_{\mathcal{H}} = \sum_{s,a}w(s,a)\left(V(s)^2+ x(s,a)^2 + h(s,a)^2 \right).
\]
The corresponding norm in $\mathcal{H}$ is denoted by $\|\cdot\|_{\mathcal{H}}$.
Define
\begin{equation}\label{comp-func-non}
 \begin{split}
 \tilde{\mathcal{L}}^k_{\beta,\mu}(V,h,x) 
&:= \sum_{s} \rho_0(s)V(s) +\frac{1}{2\mu}\sum_{s,a}w(s,a)x(s,a)Z_\mu(V,h,x_{\theta^{k}},s,a)\\
&+\frac{\beta}{2}\sum_{s,a} w(s,a)\left[x(s,a) - \left(Z_\mu(V,h,x_{\theta^{k}},s,a)\right)\right]^2.
\end{split}
\end{equation}
Here, we use the same notation as \eqref{comp-func}. It is easy to distinguish them according to variables. Let $\tilde{V}^k, \tilde{h}^k, \tilde{x}^k = \arg\min_{h\geq 0, V, x} \tilde{\mathcal{L}}^k_{\beta,\mu}(V,h,x)$. The following lemma shows properties of $\tilde{\mathcal{L}}^k_{\beta,\mu}(V,h,x)$. 
\begin{lemma}\label{lem:cvx}
    If $\beta > \frac{1}{4\mu}$, the following claims hold.
    \begin{enumerate}[(i)]
        \item $\tilde{\mathcal{L}}^k_{\beta,\mu}(V,h,x)$ is a restricted strongly convex function with respect to $V,h,x$, i.e., there exits a constant $c$ such that
        \bee
            \tilde{\mathcal{L}}^k_{\beta,\mu}(V,h,x) - \tilde{\mathcal{L}}^k_{\beta,\mu}(\tilde{V}^k, \tilde{h}^k, \tilde{x}^k) \geq \frac{c}{2}\dist\left((V,h,x),\mathcal{X}^*\right)^2,  
        \eee
        where $\mathcal{X}^*$ is the set of optimal solutions of $\tilde{\mathcal{L}}^k_{\beta,\mu}(V,h,x)$.
        The gradient of $\tilde{\mathcal{L}}^k_{\beta,\mu}(V,h,x)$ is $L$-Lipschitz continuous. 
        \item $\tilde{\mathcal{L}}^k_{\beta,\mu}(V,h,x)$ is a $\frac{1}{2\mu}(1-\frac{1}{4\mu\beta})$-stongly convex function with respect to $x$ for fixed $V,h$.
    \end{enumerate}
\end{lemma}
\begin{proof}
    The function $\tilde{\mathcal{L}}^k_{\beta,\mu}(V,h,x)$ is a quadratic function. Hence, the gradient of $\tilde{\mathcal{L}}^k_{\beta,\mu}(V,h,x)$ is $L$-Lipschitz continuous. To show the convexity of $\tilde{\mathcal{L}}^k_{\beta,\mu}(V,h,x)$, we only need to show that the quadratic part is nonnegtive.  If $\beta > \frac{1}{4\mu}$, by rearranging the quadratic term in $\tilde{\mathcal{L}}^k_{\beta,\mu}(V,h,x)$, we obtain
    \bee
 \begin{split}
& \frac{1}{2\mu}\sum_{s,a}w(s,a)x(s,a)Z_\mu(V,h,x_{\theta^{k}},s,a)
+\frac{\beta}{2}\sum_{s,a} w(s,a)\left[x(s,a) - \left(Z_\mu(V,h,x_{\theta^{k}},s,a)\right)\right]^2 \\
=& \frac{\beta}{2}\sum_{s,a}w(s,a) \left\{ \left[(1-\frac{1}{2\mu \beta})x(s,a) -  Z_\mu(V,h,x_{\theta^{k}},s,a) \right]^2 + \frac{1}{4\mu\beta}(1-\frac{1}{4\mu\beta})x(s,a)^2 \right\} \\
\geq & \frac{1}{2\mu}(1-\frac{1}{4\mu\beta})\sum_{s,a}w(s,a)x(s,a)^2 \geq 0.
\end{split}
\eee
Hence, $\tilde{\mathcal{L}}^k_{\beta,\mu}(V,h,x)$ is restricted strongly convex and $c$ is the smallest positive eigenvalue of the Hessian of $\tilde{\mathcal{L}}^k_{\beta,\mu}(V,h,x)$. It also implies that $\tilde{\mathcal{L}}^k_{\beta,\mu}(V,h,x)$ is $\frac{1}{2\mu}(1-\frac{1}{4\mu\beta})$-strongly convex with respect to $x$.
\end{proof}

We make the following assumption about the optimal points $\tilde{V}^k, \tilde{h}^k, \tilde{x}^k$ and the distribution $w(s,a)$. 
 
\begin{assumption}\label{assum:fc}
    \begin{enumerate}[(i)]
    \item It holds that $\tilde{V}^k \in \mathcal{F}^V_{R_V, m_V}, \tilde{h}^k \in \mathcal{F}^h_{R_h, m_h}, \tilde{x}^k \in \mathcal{F}^x_{R_x, m_x}$ for any $k$.
    \item For any unit vector $u, v$ and nonnegtive scalar $\tau$, $E_w 1_{\{|u^T(s,a)| < \tau\}} \leq c\tau$ and $E_w 1_{\{|v^Ts| < \tau\}} \leq c\tau$ holds, where $c$ is a constant and $E_w f = \sum_{s,a} w(s,a)f(s,a)$.  
\end{enumerate}
\end{assumption}
Assumption \ref{assum:fc} (i) is commonly used in the literature and Assumption \ref{assum:fc} (ii) is on the regularity of the distribution $w$. For convenience, let $\zeta = (\phi, \psi, \theta)$, $\Gamma = (V,h,x)$, $R = \max\{R_V, R_h, R_x\}$ and $m = \min\{m_V, m_h, m_x\}$. The following method is used to minimize \eqref{comp-func}.
Let $\zeta^{k,0} = \zeta^{k}$ and 
\be\label{eq:pjg}
\begin{split}
    \zeta^{k,t+1} &= \Pi_{\zeta \in \Xi}\left(\zeta^{k,t} - \alpha_t \nabla_\zeta \tilde{\mathcal{L}}^k_{\beta,\mu}(V_{\phi^{k,t}},h_{\psi^{k,t}},x_{\theta^{k,t}} )\right), \\
    \zeta^{k+1} &= \frac{1}{T}\sum_{t=1}^T \zeta^{k,t}.
\end{split}
\ee
The scheme \eqref{eq:pjg} is actually a projected gradient descent method. Comparing with the algorithm in Section \ref{The Augmented Lagrangian Method}, the scheme \eqref{eq:pjg} has an additional projection and average operator, which is easy  for analysis. We also ignore the stochastic part, which is not hard to analyze by the method in \cite{liu2019neural}.
In the following part, we show the convergence of the scheme \eqref{eq:pjg}. Define the linearization function
\beaa
V_{\phi}^0 (s) &=& \frac{1}{\sqrt{m_V}} \sum_{i=1}^{m_V} \phi_{1,i} \cdot 1_{\{(\phi_{2,i}^0)^Ts \geq 0 \}} \phi_{2,i}^Ts, \\ 
h_{\psi}^0 (s,a) &=& \frac{1}{\sqrt{m_h}} \sum_{i=1}^{m_h} \psi_{1,i} \cdot 1_{\{(\psi_{2,i}^0)^T(s,a) \geq 0 \}} \psi_{2,i}^T(s,a), \\
x_{\theta}^0 (s,a) &=& \frac{1}{\sqrt{m_x}} \sum_{i=1}^{m_x} \theta_{1,i} \cdot 1_{\{(\theta_{2,i}^0)^T(s,a) \geq 0 \}} \theta_{2,i}^T(s,a).
\eeaa
The following lemma gives the local linearization property. 
\begin{lemma}\label{lem:lr}
    Suppose that Assumption \ref{assum:fc} holds and $\zeta \in \Xi$. 
    Then we have the following estimation:
    \beaa
    E_{0,w}(V_{\phi}(s) - V_{\phi}^0(s))^2 &=& O\left(R_V^3m_V^{-1/2}\right), \\
    E_{0,w}(h_{\psi}(s,a) - h_{\psi}^0(s.a))^2 &=& O\left(R_h^3m_h^{-1/2}\right),\\
    E_{0,w}(x_{\theta}(s,a) - x_{\theta}^0(s,a))^2 &=& O\left(R_x^3m_x^{-1/2}\right), \\
     E_{0}\left| \tilde{\mathcal{L}}^k_{\beta,\mu}(V_{\phi},h_{\psi},x_{\theta}) -  \tilde{\mathcal{L}}^k_{\beta,\mu}(V_{\phi}^0,h_{\psi}^0,x_{\theta}^0)\right| &=& O\left(R^3m^{-1/2} + R^2m^{-1/4}\right), \\ 
     E_{0}\|\nabla_\zeta \tilde{\mathcal{L}}^k_{\beta,\mu}(V_{\phi},h_{\psi},x_{\theta}) - \nabla_\zeta \tilde{\mathcal{L}}^k_{\beta,\mu}(V_{\phi}^0,h_{\psi}^0,x_{\theta}^0)\|_2^2 &=& O\left(R^3m^{-1/2}\right),
     \eeaa
     where $E_{0,w}$ represent the expectation under the distribution \eqref{eq:dist} and $w$.
\end{lemma}
\begin{proof}
The first three estimations follows from \cite[Lemma D.1]{liu2019neural}. The
main difference is that $\psi_{1,i}$ is always equal to $1$, but it does not change the arguments.  By Lemma \ref{lem:cvx}, it yields
\be\label{eq:est-g0}
\begin{split}
&  \| \nabla_{\Gamma}\tilde{\mathcal{L}}^k_{\beta,\mu}(V_{\phi}^0,h_{\psi}^0,x_{\theta}^0)\|_{\mathcal{H}}^2 \\
\leq &  2\| \nabla_{\Gamma}\tilde{\mathcal{L}}^k_{\beta,\mu}(V_{\phi^0},h_{\psi^0},x_{\theta^0})\|_{\mathcal{H}}^2
+ 2\| \nabla_{\Gamma}\tilde{\mathcal{L}}^k_{\beta,\mu}(V_{\phi}^0,h_{\psi}^0,x_{\theta}^0) - \nabla_{\Gamma}\tilde{\mathcal{L}}^k_{\beta,\mu}(V_{\phi^0},h_{\psi^0},x_{\theta^0})\|_{\mathcal{H}}^2 \\
\leq & 2\| \nabla_{\Gamma}\tilde{\mathcal{L}}^k_{\beta,\mu}(V_{\phi^0},h_{\psi^0},x_{\theta^0})\|_{\mathcal{H}}^2+ 
2L\| (V_{\phi}^0,h_{\psi}^0,x_{\theta}^0) - (V_{\phi^0},h_{\psi^0},x_{\theta^0})\|_{\mathcal{H}}^2 \\
\leq & 4L\| (V_{\phi^0},h_{\psi^0},x_{\theta^0})\|_{\mathcal{H}}^2+ 
4\| \nabla_{\Gamma}\tilde{\mathcal{L}}^k_{\beta,\mu}(0,0,0)\|_{\mathcal{H}}^2
+ 6LR,
\end{split}
\ee
where the last inequality uses the bound of $V, h, x$ in \eqref{eq:bv}.
It implies
\[
E_0  \| \nabla_{\Gamma}\tilde{\mathcal{L}}^k_{\beta,\mu}(V_{\phi}^0,h_{\psi}^0,x_{\theta}^0)\|_{\mathcal{H}}^2 = O\left(R\right).
\]
It follows from Lemma \ref{lem:cvx} and \cite[Theorem 18.15]{bauschke2011convex} that
\beaa
&& E_{0}\left| \tilde{\mathcal{L}}^k_{\beta,\mu}(V_{\phi},h_{\psi},x_{\theta}) -  \tilde{\mathcal{L}}^k_{\beta,\mu}(V_{\phi}^0,h_{\psi}^0,x_{\theta}^0)\right| \\
&\leq & E_0\left| \iprod{ \nabla_\Gamma\tilde{\mathcal{L}}^k_{\beta,\mu}(V_{\phi}^0,h_{\psi}^0,x_{\theta}^0)}{(V_{\phi},h_{\psi},x_{\theta}) - (V_{\phi}^0,h_{\psi}^0,x_{\theta}^0)}_{\mathcal{H}}\right| \\
&& + \frac{L}{2}E_0  \|(V_{\phi},h_{\psi},x_{\theta}) - (V_{\phi}^0,h_{\psi}^0,x_{\theta}^0)\|^2_{\mathcal{H}} \\
&\leq & E_0 \| \nabla_\Gamma\tilde{\mathcal{L}}^k_{\beta,\mu}(V_{\phi}^0,h_{\psi}^0,x_{\theta}^0)\|\|(V_{\phi},h_{\psi},x_{\theta}) - (V_{\phi}^0,h_{\psi}^0,x_{\theta}^0)\|_{\mathcal{H}} \\
&&+ \frac{L}{2}E_0 \|(V_{\phi},h_{\psi},x_{\theta}) - (V_{\phi}^0,h_{\psi}^0,x_{\theta}^0)\|^2_{\mathcal{H}} \\
&\leq &  \left(E_0\| \nabla_\Gamma\tilde{\mathcal{L}}^k_{\beta,\mu}(V_{\phi}^0,h_{\psi}^0,x_{\theta}^0)\|^2_{\mathcal{H}} \right)^{\frac{1}{2}}
\left( E_0\|(V_{\phi},h_{\psi},x_{\theta}) - (V_{\phi}^0,h_{\psi}^0,x_{\theta}^0)\|^2_{\mathcal{H}} \right)^{\frac{1}{2}}\\
&&+ \frac{L}{2}E_0 \|(V_{\phi},h_{\psi},x_{\theta}) - (V_{\phi}^0,h_{\psi}^0,x_{\theta}^0)\|^2_{\mathcal{H}} \\
&=& O\left(R^3m^{-1/2} + R^2m^{-1/4}\right).
\eeaa
We next prove the last estimation. We can obtain
\beaa
 && E_{0}\|\nabla_\zeta \tilde{\mathcal{L}}^k_{\beta,\mu}(V_{\phi},h_{\psi},x_{\theta}) - \nabla_\zeta \tilde{\mathcal{L}}^k_{\beta,\mu}(V_{\phi}^0,h_{\psi}^0,x_{\theta}^0)\|_2^2  \\
 &\leq & 2E_{0}\|\nabla_\zeta(V_{\phi},h_{\psi},x_{\theta})^T\cdot \left( \nabla_\Gamma \tilde{\mathcal{L}}^k_{\beta,\mu}(V_{\phi},h_{\psi},x_{\theta}) 
 - \nabla_\Gamma \tilde{\mathcal{L}}^k_{\beta,\mu}(V_{\phi}^0,h_{\psi}^0,x_{\theta}^0) \right)\|_2^2  \\
 && + 2E_{0}\|\left( \nabla_\zeta(V_{\phi},h_{\psi},x_{\theta}) 
 - \nabla_\zeta(V_{\phi}^0,h_{\psi}^0,x_{\theta}^0) \right)^T\cdot\nabla_\Gamma \tilde{\mathcal{L}}^k_{\beta,\mu}(V_{\phi}^0,h_{\psi}^0,x_{\theta}^0)\|_2^2  \\
  &\leq & 2E_{0}\|\nabla_\zeta(V_{\phi},h_{\psi},x_{\theta})\|_*^2
  \|\nabla_\Gamma \tilde{\mathcal{L}}^k_{\beta,\mu}(V_{\phi},h_{\psi},x_{\theta}) 
 - \nabla_\Gamma \tilde{\mathcal{L}}^k_{\beta,\mu}(V_{\phi}^0,h_{\psi}^0,x_{\theta}^0)\|_\mathcal{H}^2  \\
 && + 2E_{0}\|\nabla_\zeta(V_{\phi},h_{\psi},x_{\theta}) - \nabla_\zeta(V_{\phi}^0,h_{\psi}^0,x_{\theta}^0)\|_*^2
 \|\nabla_\Gamma \tilde{\mathcal{L}}^k_{\beta,\mu}(V_{\phi}^0,h_{\psi}^0,x_{\theta}^0)\|_\mathcal{H}^2,  \\
 &:=& I_1 + I_2,
\eeaa
where the second inequality is due to the triangle inequality and the H\"older inequality and $\|\cdot\|_*$ is defined by
\be\label{eq:est-g1}
\|\nabla_\zeta(V_{\phi},h_{\psi},x_{\theta})\|_*^2 = \sum_{s,a}w(s,a)\left[ \|\nabla_\phi V_\phi(s)\|^2 
+ \|\nabla_\psi h_\psi(s)\|^2 + \|\nabla_\theta x_\theta(s)\|^2\right] \leq 3.
\ee
Using the relationship \eqref{eq:est-g1}, we get 
\beaa
I_1 &\leq& 6E_{0}\|\nabla_\Gamma \tilde{\mathcal{L}}^k_{\beta,\mu}(V_{\phi},h_{\psi},x_{\theta}) 
 - \nabla_\Gamma \tilde{\mathcal{L}}^k_{\beta,\mu}(V_{\phi}^0,h_{\psi}^0,x_{\theta}^0)\|_\mathcal{H}^2 \\
 &\leq& 6L E_{0}\|(V_{\phi},h_{\psi},x_{\theta})  - (V_{\phi}^0,h_{\psi}^0,x_{\theta}^0)\|_\mathcal{H}^2 
 =  O\left(R^3m^{-1/2}\right).
\eeaa
According to the arguments in  \cite[Lemma D.2]{liu2019neural}, we have
\be\label{eq:est-g2}
\begin{split}
&E_{0}\|\nabla_\zeta(V_{\phi},h_{\psi},x_{\theta}) - \nabla_\zeta(V_{\phi}^0,h_{\psi}^0,x_{\theta}^0)\|_*^2
=  O\left(Rm^{-1/2}\right), \\
&E_{0}\|\nabla_\zeta(V_{\phi},h_{\psi},x_{\theta}) - \nabla_\zeta(V_{\phi}^0,h_{\psi}^0,x_{\theta}^0)\|_*^2
\| (V_{\phi^0},h_{\psi^0},x_{\theta^0})\|_{\mathcal{H}}^2 
=  O\left(Rm^{-1/2}\right).
\end{split}
\ee
Combining \eqref{eq:est-g0} and \eqref{eq:est-g2} yields
\beaa
I_2 
 =  O\left(R^2m^{-1/2}\right).
\eeaa
 which completes the proof.
\end{proof}

The next lemma gives  the convergence of the scheme \eqref{eq:pjg}. 
\begin{lemma}\label{lem:conv1}
    Suppose that Assumption \ref{assum:fc} holds and $\beta > \frac{1}{4\mu}$. Let  $\zeta^{k,t} = (\phi^{k,t}, \psi^{k,t}, \theta^{k,t})$ be a sequence generated by the projected gradient descent method \eqref{eq:pjg} with $\alpha_t = \alpha \leq \frac{1}{\sqrt{3}L}$.
    Then we have
    \be\label{eq:conv1}
    \begin{split}
   & E_0\left(\tilde{\mathcal{L}}^k_{\beta,\mu}(V_{\phi^{k+1}},h_{\psi^{k+1}},x_{\theta^{k+1}}) - \tilde{\mathcal{L}}^k_{\beta,\mu}(\tilde{V}^{k+1}, \tilde{h}^{k+1}, \tilde{x}^{k+1})\right) \\
    = &O\left(R^2T^{-1} + R^{5/2}m^{-1/4} + R^{3}m^{-1/2}\right).
    \end{split}
    \ee
    Furthermore, it holds that
    \be\label{eq:conv2}
    E_{0,w}(x_{\theta^{k+1}}(s,a) - \tilde{x}^{k+1}(s,a))^2 = O\left(R^2T^{-1} + R^{5/2}m^{-1/4} + R^{3}m^{-1/2}\right). 
    \ee
\end{lemma}
\begin{proof}
    Since $V_{\phi}^0 (s), h_{\psi}^0 (s,a), x_{\theta}^0 (s,a)$ is linear with respect to $\phi,\psi,\theta$, respectively, $\tilde{\mathcal{L}}^k_{\beta,\mu}(V_{\phi}^0,h_{\psi}^0,x_{\theta}^0)$ is convex for $\zeta$ by Lemma \ref{lem:cvx}. We also have
    \be
    \begin{split}
   &  \| \nabla_\zeta \tilde{\mathcal{L}}^k_{\beta,\mu}(V_{\phi}^0,h_{\psi}^0,x_{\theta}^0) - \nabla_\zeta \tilde{\mathcal{L}}^k_{\beta,\mu}(V_{\phi'}^0,h_{\psi'}^0,x_{\theta'}^0)  \|^2 \\
    \leq & \|\nabla_\zeta(V_{\phi}^0,h_{\psi}^0,x_{\theta}^0)^T\cdot 
    \left( \nabla_\Gamma \tilde{\mathcal{L}}^k_{\beta,\mu}(V_{\phi}^0,h_{\psi}^0,x_{\theta}^0) 
 - \nabla_\Gamma \tilde{\mathcal{L}}^k_{\beta,\mu}(V_{\phi'}^0,h_{\psi'}^0,x_{\theta'}^0) \right)\|^2 \\
     \leq &\|\nabla_\zeta(V_{\phi}^0,h_{\psi}^0,x_{\theta}^0) \|_*^2
     \| \nabla_\Gamma \tilde{\mathcal{L}}^k_{\beta,\mu}(V_{\phi}^0,h_{\psi}^0,x_{\theta}^0) 
 - \nabla_\Gamma \tilde{\mathcal{L}}^k_{\beta,\mu}(V_{\phi'}^0,h_{\psi'}^0,x_{\theta'}^0) \|_\mathcal{H}^2 \\
 \leq & 3L^2 \| (V_{\phi}^0,h_{\psi}^0,x_{\theta}^0) 
 - (V_{\phi'}^0,h_{\psi'}^0,x_{\theta'}^0) \|_\mathcal{H}^2 \leq 3L^2 \| \zeta - \zeta' \|_2^2.
    \end{split}
    \ee
    Hence, $\tilde{\mathcal{L}}^k_{\beta,\mu}(V_{\phi}^0,h_{\psi}^0,x_{\theta}^0)$ has  $\sqrt{3}L$-Lipschitz continuous gradient. According to \cite[Theorem 18.15]{bauschke2011convex}, it implies 
\be\label{eq:cvx1}
\begin{split}
&\tilde{\mathcal{L}}^k_{\beta,\mu}(V_{\phi^{k,t+1}}^0,h_{\psi^{k,t+1}}^0,x_{\theta^{k,t+1}}^0) - \tilde{\mathcal{L}}^k_{\beta,\mu}(V_{\phi^{k,t}}^0,h_{\psi^{k,t}}^0,x_{\theta^{k,t}}^0) \\
\leq & \nabla_\zeta \tilde{\mathcal{L}}^k_{\beta,\mu}(V_{\phi^{k,t}}^0,h_{\psi^{k,t}}^0,x_{\theta^{k,t}}^0)^T(\zeta^{k,t+1} - \zeta^{k,t} ) + \frac{\sqrt{3}L}{2}\|\zeta^{k,t+1} - \zeta^{k,t} \|_2^2.
\end{split}
\ee
By the convexity of $\tilde{\mathcal{L}}^k_{\beta,\mu}(V_{\phi}^0,h_{\psi}^0,x_{\theta}^0)$, we have
\be\label{eq:cvx2}
\begin{split}
    & \tilde{\mathcal{L}}^k_{\beta,\mu}(V_{\phi^{k,t}}^0,h_{\psi^{k,t}}^0,x_{\theta^{k,t}}^0) - \tilde{\mathcal{L}}^k_{\beta,\mu}(V_{\phi^{*}}^0,h_{\psi^{*}}^0,x_{\theta^{*}}^0)  \\
\leq & \nabla_\zeta \tilde{\mathcal{L}}^k_{\beta,\mu}(V_{\phi^{k,t}}^0,h_{\psi^{k,t}}^0,x_{\theta^{k,t}}^0)^T(\zeta^{k,t} - \zeta^{*} ),
\end{split}
\ee
where $\zeta^* = (\phi^*,\psi^*,\theta^*)$ is an optimal solution, i.e., $(V_{\phi^{*}}^0,h_{\psi^{*}}^0,x_{\theta^{*}}^0) = (\tilde{V}^{k+1}, \tilde{h}^{k+1}, \tilde{x}^{k+1})$.
According to the property of projection operator, we get
\be\label{eq:cvx3}
({\zeta^{*} - \zeta^{k,t+1}})^T\left(\zeta^{k,t} - \alpha_t \nabla_\zeta \tilde{\mathcal{L}}^k_{\beta,\mu}(V_{\phi^{k,t}},h_{\psi^{k,t}},x_{\theta^{k,t}} ) - \zeta^{k,t+1}\right) \geq 0
\ee
Combining \eqref{eq:cvx1}, \eqref{eq:cvx2} and \eqref{eq:cvx3} yields
\be\label{eq:cvx4}
\begin{split}
&\tilde{\mathcal{L}}^k_{\beta,\mu}(V_{\phi^{k,t+1}}^0,h_{\psi^{k,t+1}}^0,x_{\theta^{k,t+1}}^0) - \tilde{\mathcal{L}}^k_{\beta,\mu}(V_{\phi^{*}}^0,h_{\psi^{*}}^0,x_{\theta^{*}}^0) \\
\leq & \nabla_\zeta \tilde{\mathcal{L}}^k_{\beta,\mu}(V_{\phi^{k,t}}^0,h_{\psi^{k,t}}^0,x_{\theta^{k,t}}^0)^T(\zeta^{k,t+1} - \zeta^{*} ) + \frac{\sqrt{3}L}{2}\|\zeta^{k,t+1} - \zeta^{k,t} \|_2^2 \\
\leq & \frac{1}{\alpha_t}(\zeta^{k,t}  - \zeta^{k,t+1})^T(\zeta^{k,t+1} - \zeta^{*} ) + \frac{\sqrt{3}L}{2}\|\zeta^{k,t+1} - \zeta^{k,t} \|_2^2 \\ 
&+ \left(\nabla_\zeta \tilde{\mathcal{L}}^k_{\beta,\mu}(V_{\phi^{k,t}}^0,h_{\psi^{k,t}}^0,x_{\theta^{k,t}}^0) -  \nabla_\zeta \tilde{\mathcal{L}}^k_{\beta,\mu}(V_{\phi^{k,t}},h_{\psi^{k,t}},x_{\theta^{k,t}} )  \right)^T(\zeta^{k,t+1} - \zeta^{*} ) \\
\leq& \frac{1}{\alpha}\left(\|\zeta^{k,t} - \zeta^{*}\|^2 - \|\zeta^{k,t+1} - \zeta^{*}\|^2\right) \\
&+ \|\nabla_\zeta \tilde{\mathcal{L}}^k_{\beta,\mu}(V_{\phi^{k,t}}^0,h_{\psi^{k,t}}^0,x_{\theta^{k,t}}^0) -  \nabla_\zeta \tilde{\mathcal{L}}^k_{\beta,\mu}(V_{\phi^{k,t}},h_{\psi^{k,t}},x_{\theta^{k,t}} )\|\|\zeta^{k,t+1} - \zeta^{*}\|.
\end{split}
\ee
By the convexity of $\tilde{\mathcal{L}}^k_{\beta,\mu}(V_{\phi^{k,t+1}}^0,h_{\psi^{k,t+1}}^0,x_{\theta^{k,t+1}}^0)$ and Lemma \ref{lem:lr}, we obtain
\bee
    \begin{split}
     & E_0\left(\tilde{\mathcal{L}}^k_{\beta,\mu}(V_{\phi^{k+1}}^0,h_{\psi^{k+1}}^0,x_{\theta^{k+1}}^0) - \tilde{\mathcal{L}}^k_{\beta,\mu}(\tilde{V}^{k+1}, \tilde{h}^{k+1}, \tilde{x}^{k+1})\right) \\
     \leq & \frac{1}{T} \sum_{t=1}^T E_0\left(\tilde{\mathcal{L}}^k_{\beta,\mu}(V_{\phi^{k,t}}^0,h_{\psi^{k,t}}^0,x_{\theta^{k,t}}^0) - \tilde{\mathcal{L}}^k_{\beta,\mu}(\tilde{V}^{k+1}, \tilde{h}^{k+1}, \tilde{x}^{k+1})\right) \\
 \leq &  \frac{4R}{T} \sum_{t=1}^T \left( E_{0}\|\nabla_\zeta \tilde{\mathcal{L}}^k_{\beta,\mu}(V_{\phi^{k,t}}^0,h_{\psi^{k,t}}^0,x_{\theta^{k,t}}^0) -  \nabla_\zeta \tilde{\mathcal{L}}^k_{\beta,\mu}(V_{\phi^{k,t}},h_{\psi^{k,t}},x_{\theta^{k,t}} )\|^2\right)^{\frac{1}{2}}  \\
 &+ \frac{1}{\alpha T} \|\zeta^{k,0} - \zeta^{*}\|^2\\
 = & O\left(R^2T^{-1} + R^{5/2}m^{-1/4} \right),
\end{split}
\eee
where the second inequality is due to \eqref{eq:cvx4} and H\"older inequality.
It follows from Lemma \ref{lem:lr} that
\bee
    \begin{split}
&E_0\left(\tilde{\mathcal{L}}^k_{\beta,\mu}(V_{\phi^{k+1}},h_{\psi^{k+1}},x_{\theta^{k+1}}) - \tilde{\mathcal{L}}^k_{\beta,\mu}(\tilde{V}^{k+1}, \tilde{h}^{k+1}, \tilde{x}^{k+1})\right) \\
\leq & E_0\left(\tilde{\mathcal{L}}^k_{\beta,\mu}(V_{\phi^{k+1}}^0,h_{\psi^{k+1}}^0,x_{\theta^{k+1}}^0) - \tilde{\mathcal{L}}^k_{\beta,\mu}(\tilde{V}^{k+1}, \tilde{h}^{k+1}, \tilde{x}^{k+1})\right) \\
&+ E_0\left(\tilde{\mathcal{L}}^k_{\beta,\mu}(V_{\phi^{k+1}},h_{\psi^{k+1}},x_{\theta^{k+1}}) - \tilde{\mathcal{L}}^k_{\beta,\mu}(V_{\phi^{k+1}}^0,h_{\psi^{k+1}}^0,x_{\theta^{k+1}}^0)\right) \\
&= O\left(R^2T^{-1} + R^{5/2}m^{-1/4} + R^{3}m^{-1/2}\right),
\end{split}
\eee
which completes the proof of \eqref{eq:conv1}.

By Lemma \ref{lem:cvx} (i), we have
\bee
    E_0\dist\left(\left(V_{\phi^{k+1}},h_{\psi^{k+1}},x_{\theta^{k+1}}\right),\mathcal{X}^*\right)^2 = O\left(R^2T^{-1} + R^{5/2}m^{-1/4} + R^{3}m^{-1/2}\right).
\eee
It means that there exists an optimal solution $(V^*, h^*, x^*) \in \mathcal{X}^*$ such that
\bee
    E_{0,w}(x_{\theta^{k+1}}(s,a) - {x}^*(s,a))^2 = O\left(R^2T^{-1} + R^{5/2}m^{-1/4} + R^{3}m^{-1/2}\right).
\eee
By Lemma \ref{lem:cvx} (ii), $x^*$ is unique in $\mathcal{X}^*$, i.e. $x^* = \tilde{x}^{k+1}$, which completes the proof of \eqref{eq:conv2}.
\end{proof}

The next lemma shows an estimation about the error between the quadratic penalty method and augmented Lagrangian method.
\begin{lemma}\label{lem:err-2}
    Suppose that Assumption \ref{assum:fc} holds and $\beta > \frac{1}{4\mu}$.
    Define $\hat{V}^{k},\hat{h}^{k} = \arg\min_{h\geq 0,V}  \mathcal{L}_{\mu}(V,h, x_{\theta^k})$ and let 
    $
        \hat{x}^{k+1}(s,a) = [Z_\mu( \hat{V}^{k}, \hat{h}^k, x_{\theta^k},s,a)].
    $
    Then we have
    \bea
    Z_\mu(\tilde{V}^{k+1}, \tilde{h}^k, x_{\theta^{k}}, s,a) &=& \xi_1 Z_\mu( \hat{V}^{k+1}, \hat{h}^{k+1}, x_{\theta^k},s,a), \\
    \tilde{x}^{k+1} &=& \xi_2 \hat{x}^{k+1},
    \eea
    where $\xi_1 = \frac{4\mu\beta}{4\mu\beta-1}$ and $\xi_2 = \frac{4\mu\beta - 2}{4\mu\beta-1}$.
    Furthermore, it implies
    \[
        E_{0,w}(x_{\theta^{k+1}}(s,a) - \hat{x}^{k+1}(s,a))^2 = O\left(R^2T^{-1} + R^{5/2}m^{-1/4} + R^{3}m^{-1/2} + \beta^{-2}R\right).
    \]
\end{lemma}

\begin{proof}
    Let $x^*_{V,h} = \arg\min_x \tilde{\mathcal{L}}^k_{\beta,\mu}(V,h,x)$. The first order optimality condition gives 
    \[
        \frac{1}{2\mu}Z_\mu(V,h,x_{\theta^{k}},s,a)
        + \beta\left[x^*_{V,h}(s,a) - \left(Z_\mu(V,h,x_{\theta^{k}},s,a)\right)\right] = 0.
    \]
    It implies that 
    \be\label{eq:xvh}
    x^*_{V,h}(s,a) = (1-\frac{1}{2\mu\beta})Z_\mu(V,h,x_{\theta^{k}},s,a).
\ee
Define
    \[
    \mathcal{M}(V,h) := \tilde{\mathcal{L}}^k_{\beta,\mu}(V,h,x^*_{V,h}) 
    = \sum_{s} \rho_0(s)V(s) +\frac{1}{2\mu\xi_1}\sum_{s,a}w(s,a)[Z_\mu(V,h,x_{\theta^{k}},s,a)]^2.
    \]
    The optimality condition of  $\min_{h \geq 0, V} \mathcal{M}(V,h)$ is
\bee
    \begin{split}
        &\xi_1 \rho_0(s') - \sum_{s,a}w(s,a)Z_\mu(V^*,h^*,x',s,a)(\delta_{s'}(s) - \gamma P(s'|s,a)) = 0, \forall s' \\
    & \mu h(x,a) = \left[- x'(s,a) - \mu\left(r(s,a)+\sum_{s'} \gamma P(s'|s,a) V(s') - V(s)\right)\right]_+.
\end{split}
\eee
Together with \eqref{eq:kkt-lg}, we have 
\beaa
Z_\mu(\tilde{V}^{k+1}, \tilde{h}^k, x_{\theta^{k}}, s,a) &=& \xi_1 Z_\mu(\hat{V}^{k+1}, \hat{h}^{k+1},x_{\theta^k},  s,a). 
    \eeaa
    Combining with \eqref{eq:xvh}, we have
\[
    \tilde{x}^{k+1}(s,a) = x^*_{\tilde{V}^{k+1},\tilde{h}^{k+1}}(s,a) = (1-\frac{1}{2\mu\beta})Z_\mu(\tilde{V}^{k+1},\tilde{h}^{k+1},x_{\theta^{k}},s,a) = \xi_2 \hat{x}^{k+1}(s,a).
\]
According to Lemma \ref{lem:conv1}, we have
\beaa
   &&  E_{0,w}(x_{\theta^{k+1}}(s,a) - \hat{x}^{k+1}(s,a))^2 
\\
   &\leq &  2E_{0,w}(x_{\theta^{k+1}}(s,a) - \tilde{x}^{k+1}(s,a))^2 +  2E_{0,w}(\tilde{x}^{k+1}(s,a) - \hat{x}^{k+1}(s,a))^2 \\
   &= &  2E_{0,w}(x_{\theta^{k+1}}(s,a) - \tilde{x}^{k+1}(s,a))^2 +  2(1-\xi_2)^2E_{0,w}( \hat{x}^{k+1}(s,a))^2 \\
   & = & O\left(R^2T^{-1} + R^{5/2}m^{-1/4} + R^{3}m^{-1/2} + \beta^{-2}R\right)
.
\eeaa
This completes the proof of Lemma \ref{lem:err-2}.
\end{proof}

Let $\mathcal{S}$ be the subgradient operator of $-\sum_{s,a} r(s,a)w(s,a)x(s,a) + 1_\Omega(x)$ with respect to $x$, where
\bee
\Omega = \{x :  \sum_{s,a} (\delta_{s'}(s) - \gamma P(s'|s,a)) w(s,a)x(s,a) = \rho_0(s'), \forall s', 
x(s,a) \geq 0, \forall s,a\}.
\eee
Then $\mathcal{S}$  is a maximal monotone operator. Lemma \ref{lem:lp} means that 
\[
Z_\mu(\hat{V}^{k}, \hat{h}^k, x_{\theta^k}, s,a) = \mathcal{T}x_{\theta^k}:= (I + \mu \mathcal{S})^{-1}x_{\theta^k} ,
\]
where $\hat{V}^{k},\hat{h}^{k} = \arg\min_{h\geq 0,V}  \mathcal{L}_{\mu}(V,h, x_{\theta^k}) $.
We next give the main theorem. It states that the residual can be controlled if the parameters are chosen properly. 
\begin{theorem}
    Suppose that Assumption \ref{assum:fc} holds and $\beta > \frac{1}{4\mu}$. Let  $\zeta^{k+1} = (\phi^{k+1}, \psi^{k+1}, \theta^{k+1})$ be a sequence generated by \eqref{eq:pjg}. 
    Then we have that
    \bee
    \begin{split}
       & \frac{1}{K} \sum_{k=1}^K E_{0,w} \left({x}_{\theta^{k}}(s,a) - \mathcal{T}{x}_{\theta^{k}}(s,a)\right)^2 \\
        =&  O\left(R^{3/2}T^{-1/2} + R^{7/4}m^{-1/8} + R^{2}m^{-1/4} + \beta^{-1}R + K^{-1}R\right),
        \end{split}
    \eee
    and 
    \bee
\begin{split}
 &  E_{0,w}\left(x_{\theta^{K}}(s,a) - \mathcal{T}x_{\theta^{K}}(s,a)\right)^2  \\
= &O\left(KR^{3/2}T^{-1/2} + KR^{7/4}m^{-1/8} + KR^{2}m^{-1/4} + K\beta^{-1}R+ K^{-1}R \right).
\end{split}
\eee
\end{theorem}
\begin{proof}
According to Lemmas \ref{lem:lp} and \ref{lem:err-2}, we have   
\bee
    E_{0,w}\left(x_{\theta^{k+1}}(s,a) - \mathcal{T}x_{\theta^{k}}(s,a)\right)^2 = O\left(R^2T^{-1} + R^{5/2}m^{-1/4} + R^{3}m^{-1/2} + \beta^{-2}R\right). 
\eee
Let $x^*$ be the solution of  the equation $x^* = \mathcal{T}x^*$, i.e., the optimal solution of the LP \eqref{eq:rllp-dual}. It follows from \cite[Proposition 1]{rockafellar1976augmented} that 
\bee
\begin{split}
& E_{0,w}\left(x^*(s,a) - \mathcal{T}x_{\theta^{k}}(s,a)\right)^2 + E_{0,w}\left(x_{\theta^{k}}(s,a) - \mathcal{T}x_{\theta^{k}}(s,a)\right)^2 \\
\leq & E_{0,w}\left(x_{\theta^{k}}(s,a) - x^*(s,a)\right)^2.
\end{split}
\eee
It implies
\bee
\begin{split}
&  E_{0,w}\left(x_{\theta^{k}}(s,a) - \mathcal{T}x_{\theta^{k}}(s,a)\right)^2 \\
\leq & E_{0,w}\left(x_{\theta^{k}}(s,a) - x^*(s,a)\right)^2 - E_{0,w}\left(x^*(s,a) - \mathcal{T}x_{\theta^{k}}(s,a)\right)^2 \\
= & E_{0,w}\left(x_{\theta^{k}}(s,a) - x^*(s,a)\right)^2  - E_{0,w}\left(x_{\theta^{k+1}}(s,a) - x^*(s,a)\right)^2\\
& +E_{0,w}\left(x_{\theta^{k+1}}(s,a) - x^*(s,a)\right)^2 - E_{0,w}\left(x^*(s,a) - Tx_{\theta^{k}}(s,a)\right)^2 \\
= & E_{0,w}\left(x_{\theta^{k}}(s,a) - x^*(s,a)\right)^2  - E_{0,w}\left(x_{\theta^{k+1}}(s,a) - x^*(s,a)\right)^2\\
& +E_{0,w}\left(x_{\theta^{k+1}}(s,a) - \mathcal{T}x_{\theta^{k}}(s,a)\right)\left(x_{\theta^{k+1}}(s,a) +  \mathcal{T}x_{\theta^{k}}(s,a) - 2x^*(s,a)\right) \\
\leq & E_{0,w}\left(x_{\theta^{k}}(s,a) - x^*(s,a)\right)^2  - E_{0,w}\left(x_{\theta^{k+1}}(s,a) - x^*(s,a)\right)^2\\
& + cR^{1/2} \cdot (E_{0,w}\left(x_{\theta^{k+1}}(s,a) - \mathcal{T}x_{\theta^{k}}(s,a)\right)^2)^{\frac{1}{2}}, 
\end{split}
\eee
where $c$ is a constant.
Taking summation over $k$, we obtain
\bee
\begin{split}
& \frac{1}{K} \sum_{k=1}^K E_{0,w}\left(x_{\theta^{k}}(s,a) - \mathcal{T}x_{\theta^{k}}(s,a)\right)^2 \\ 
\leq &  \frac{1}{K} E_{0,w}\left(x_{\theta^{1}}(s,a) - x^*(s,a)\right)^2 + cR^{1/2} \cdot \frac{1}{K} \sum_{k=1}^K  (E_{0,w}\left(x_{\theta^{k+1}}(s,a) - \mathcal{T}x_{\theta^{k}}(s,a)\right)^2)^{\frac{1}{2}} \\
 = & O\left(R^{3/2}T^{-1/2} + R^{7/4}m^{-1/8} + R^{2}m^{-1/4} + \beta^{-1}R + K^{-1}R\right).
 \end{split}
\eee
Furthermore, by the same arguments, we have
\bee
\begin{split}
& E_{0,w}\left(x_{\theta^{k+1}}(s,a) - \mathcal{T}x_{\theta^{k+1}}(s,a)\right)^2 -  E_{0,w}\left(x_{\theta^{k}}(s,a) - \mathcal{T}x_{\theta^{k}}(s,a)\right)^2 \\
= & O\left(R^{3/2}T^{-1/2} + R^{7/4}m^{-1/8} + R^{2}m^{-1/4} + \beta^{-1}R \right).
\end{split}
\eee
It implies that
\bee
\begin{split}
 & K \cdot E_{0,w}\left(x_{\theta^{K}}(s,a) - \mathcal{T}x_{\theta^{K}}(s,a)\right)^2  \\
\leq & \sum_{k=1}^K E_{0,w}\left(x_{\theta^{k}}(s,a) - \mathcal{T}x_{\theta^{k}}(s,a)\right)^2 \\
&+  \sum_{k=1}^K (K-k) \cdot O\left(R^{3/2}T^{-1/2} + R^{7/4}m^{-1/8} + R^{2}m^{-1/4} + \beta^{-1}R \right).
\end{split}
\eee
Hence, we have
\bee
\begin{split}
 &  E_{0,w}\left(x_{\theta^{K}}(s,a) - \mathcal{T}x_{\theta^{K}}(s,a)\right)^2  \\
\leq & \frac{1}{K}\sum_{k=1}^K E_{0,w}\left(x_{\theta^{k}}(s,a) - \mathcal{T}x_{\theta^{k}}(s,a)\right)^2 \\
& + K \cdot O\left(R^{3/2}T^{-1/2} + R^{7/4}m^{-1/8} + R^{2}m^{-1/4} + \beta^{-1}R \right) \\
= & O\left(KR^{3/2}T^{-1/2} + KR^{7/4}m^{-1/8} + KR^{2}m^{-1/4} + K\beta^{-1}R+ K^{-1}R \right),
\end{split}
\eee
which completes the proof.
\end{proof}

\section{Implementation Matters}\label{implementation matters}
In this section, we explain the structure of the neural networks and clarify some strategies in our implementation.
\subsection{Network Architecture}
We describe the network structures for the value network $V_\phi$, the slack network $h_{\psi}$ and the multiplier network $x_{\theta}$ in this subsection. By definition, the value network reads state information and outputs the corresponding predicted values, while both the multiplier network and slack network take the state as input and predict the values on all possible actions. 

For the environments with continuous action spaces, a 3-layer fully connected network of 64 units using $\tanh$ activator is considered for function approximation in value network $V_{\psi}$. The multiplier network $x_{\theta}$ starts from a 2-layer MLP with 64 units and $\tanh$ activator, then splits into two heads. The first head generates a positive state value, i.e., $x_{\theta,1}(s)$, after a single fully connected layer and the second head similarly predicts a mean vector $x_{\theta,2}(s)$ which has the same dimension as action space. The Gaussian policy is commonly adopted, however in practice, the actions are bounded in a finite interval, i.e., $\mathcal{A}\subseteq[-\bar{a},\bar{a}]^D$ and $D$ is the dimension of the action space. Therefore, we apply $\tanh$ function to the Gaussian samples $u\sim\mathcal{N}(x_{\theta,2}(s),\sigma^2)$, i.e., $a =\bar{a} \tanh(u)$, and derive the log-likelihood by employing the change of
variables formula for the bounded actions. The multiplier is composed as 
\be
x_{\theta}(s,a) =x_{\theta,1}(s)\exp(\log f(u|s) - 1^T\log(1-\tanh^2(u))),
\ee
where $f(u|s)$ is the Gaussian density function and $\sigma>0$ is a hyper-parameter. As $h_{\psi}(s,a)\geq0,\forall s,a$, we take a fully connected slack network to map the state $s$ into a vector $h_{\psi}(s)$ with the same dimension as actions. The slack function is 
\be
h_{\psi}(s,a) =C(1-\exp(-\|a-h_{\psi}(s)\|^2)).
\ee
where $C>0$ is a hyper-parameter.

In discrete cases, e.g. Acrobot, CartPole, the difference from the continuous environments mainly locates in the output layer of the second head in the multiplier network, as well as in the slack network. The head directly maps the hidden variable into an action distribution by a $softmax$ activator. In other words, $x_{\theta,2}(s)$ is a vector which has the same dimension as the action space and 
\be
x_{\theta}(s,a) = x_{\theta,1}(s)x_{\theta,2}(s)^T1_a.
\ee
where $1_a$ is an indicator vector. Similarly, the slack network predicts
 \be
h_{\psi}(s,a)=C h_{\psi}(s)^T1_a.
\ee
 The network structure of the multiplier function $x_{\theta}$ is illustrated in Fig \ref{fig:netstruct}. Additionally, if the states are not presented as a vector in the Euclidean space, but video frames for example, then the fully connected layers are concatenated with multiple convolutional layers in advance to embed the states. 
\begin{figure}[h]
    \centering
    \setlength{\belowcaptionskip}{10pt}

    \tikzstyle{cir} = [circle, minimum width = 0.7cm, minimum height = 0.7cm, text centered, draw = black]
    \tikzstyle{rec} = [rectangle, minimum width =1.2cm, minimum height = 1.2cm, text centered, draw = black] 
    \tikzstyle{rec1} = [rectangle,  minimum width = 2.75cm, minimum height = 2.5cm, text centered, draw = black] 
    \tikzstyle{arrow} = [thick, ->, draw = black]
  
    \begin{tikzpicture}[node distance = 1cm]
    \node(input){state $s$};
        \node(FC)[rec1, right of = input,xshift = 2.2cm] {fully connected layers  or CNN};
        \node(hid1)[rec, right of = FC, xshift = 3.8cm,yshift = 1.2cm] {fully connected layers};
              \node(hid2)[rec, right of = FC, xshift = 3.8cm,yshift = -1.2cm] {fully connected layers};
        \node(x1)[ right of = hid1, xshift = 1.8cm] {$x_{\theta,1}(s)$};
              \node(x2)[right of =hid2, xshift = 1.8cm] {$x_{\theta,2}(s)$};

      
              \draw[arrow] (input)--(FC);
         \draw[arrow] (FC)--(hid1);
        \draw[arrow] (FC)--(hid2);
         \draw[arrow] (hid1)--(x1);
          \draw[arrow] (hid2)--(x2);

         \end{tikzpicture}
             \caption{The multiplier network $x_{\theta}$ with two heads.}
         \label{fig:netstruct}
\end{figure}
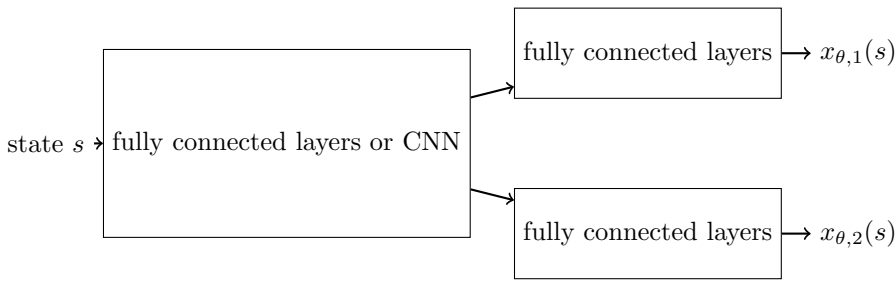

\subsection{Experimental Strategies}For practical computation considerations, we also include several helpful and widely used strategies as follows.
\begin{itemize}
\item \textbf{Proportional Sampling.} In the above discussion, we propose a stochastic gradient-based algorithm with a series of approximations to resolve the constrained LP problem. To enhance the control of the constraints in the dual LP, we are supposed to encourage the points that violate the constraints. In other words, when getting batch data from the replay buffer $\mathcal{D}$, the transition tuple $(s,a,r,s')$ is sampled proportional to the empirical constraint violation, i.e., $[r+\gamma V_{\phi^{\mathrm{targ}}}(s')-V_{\phi}(s)]_+$ where $\phi$ is the current parameter of the value network, other than uniformly sampling in general case.

\item \textbf{Adam Learning Rate Annealing.} Based on Adam's \cite{kingma2014adam} excellent performance in many stochastic optimization problems, we adopt it as our method for all networks. The learning rate of Adam is annealed as the iteration going on.

\item \textbf{Local Gradient Clipping.}  After each gradient computation, we clip each component of the gradient (i.e., the gradient of each layer) such that the local gradient norm does not exceed one. DQN implemented in OpenAI Baselines \cite{baselines} also executes such operation. 

\end{itemize}

\section{Experiments}\label{Numerical Experiments}
To get a better understanding of the stability and efficiency of our composite augmented Lagrangian algorithm, we make a comprehensive numerical experiments including a series of explorations on algorithmic components and comparisons with some state-of-the-art algorithms. As the related works \cite{dai2017boosting,cho2017deep,wang2017randomized} do not provide open-source code, we are not able to compare with them fairly. However, according to their reported results, the performance of our method is competitive with them. 

The interaction environments are provided from OpenAI Gym benchmark \cite{brockman2016openai}. As illustrated in \cite{henderson2017deep}, the implementation of the RL algorithms, as well as the preprocessing of the environments, have a large affect on numerical performance. For fairness, our method is implemented in OpenAI Baselines \cite{baselines} toolkit which provides sufficient wrappers for all type of environments. For all figures below, the solid line reports the average values and the shade area corresponds to the standard derivation over five random seeds. Besides, the horizontal axis represents the number of the interactions with the environment.

\subsection{Ablation Study}
To investigate the algorithmic components and the scalability of SCAL, we perform ablation study in several control problems by detecting two type of gradients of the weighted augmented Lagrangian function and varying the number of rollout steps.

\begin{itemize}

\item  \textbf{The Effect of Gradient Estimation.} To understand the capability of the proposed method in section \ref{A Strategy to Overcome Double-Sampling} for the double sampling obstacle, we compare the un-bias gradient estimation of $\mathcal{L}_{\mu,w}(\phi,\theta^k)$, i.e.,
\bea\notag
g_1 = \nabla_{\phi}V_{\phi^k}(s_{0})+x_{\theta^k}(s,a)\left(\gamma \nabla_{\phi}V_{\phi^k}(s')- \nabla_{\phi}V_{\phi^k}(s)\right),
\eea
 with a simple gradient approximation, i.e., 
\bea\notag
g_2 =\nabla_{\phi}V_{\phi^k}(s_{0})+z_\mu(\phi^k,\psi^k,\theta^k,s,a,s')\left(\gamma \nabla_{\phi}V_{\phi^k}(s')- \nabla_{\phi}V_{\phi^k}(s)\right),
\eea
in which the expectation $Z_\mu(\phi^k,\psi^k,\theta^k,s,a)$ is replaced by single transition sample. We call the corresponding gradient as $\textbf{un-bias}$ and $\textbf{bias}$ version, respectively. We investigate the variance of these two different gradients in a classic control task Acrobot in Fig \ref{fig:if-biasgrad-det}. Notably, the default transition dynamic in Acrobot itself is deterministic. In this case, $g_2$ is exactly an un-bias gradient estimation of $\mathcal{L}_{c,\mu}(\phi,\theta^k)$, and the left picture shows that the variance curves of these two gradient estimations are similar. Furthermore, we add a small white noise disturbance to the transition probability and the variance of two gradients are shown in the right picture. In this time, a significant variance reduction illustrates that our algorithm works well in solving the double-sampling obstacle. Moreover, for the environments in which state space is large or/and the transition kernel is stochastic, the $\textbf{bias}$ version may also induce a horrible deviation to mislead the optimization.
\begin{figure}[tbhp]
\centering
\includegraphics[scale = 0.36]{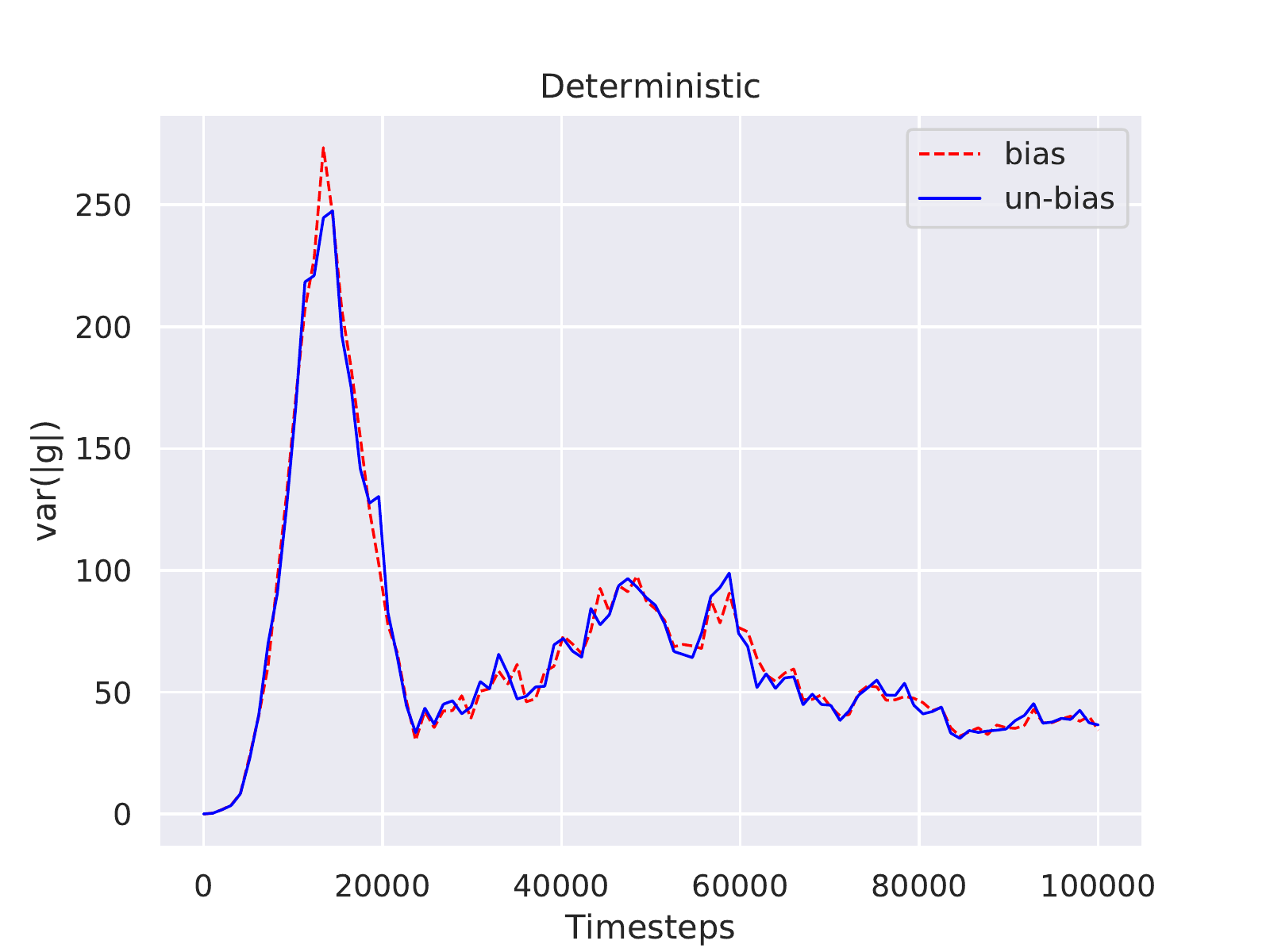}
\includegraphics[scale = 0.36]{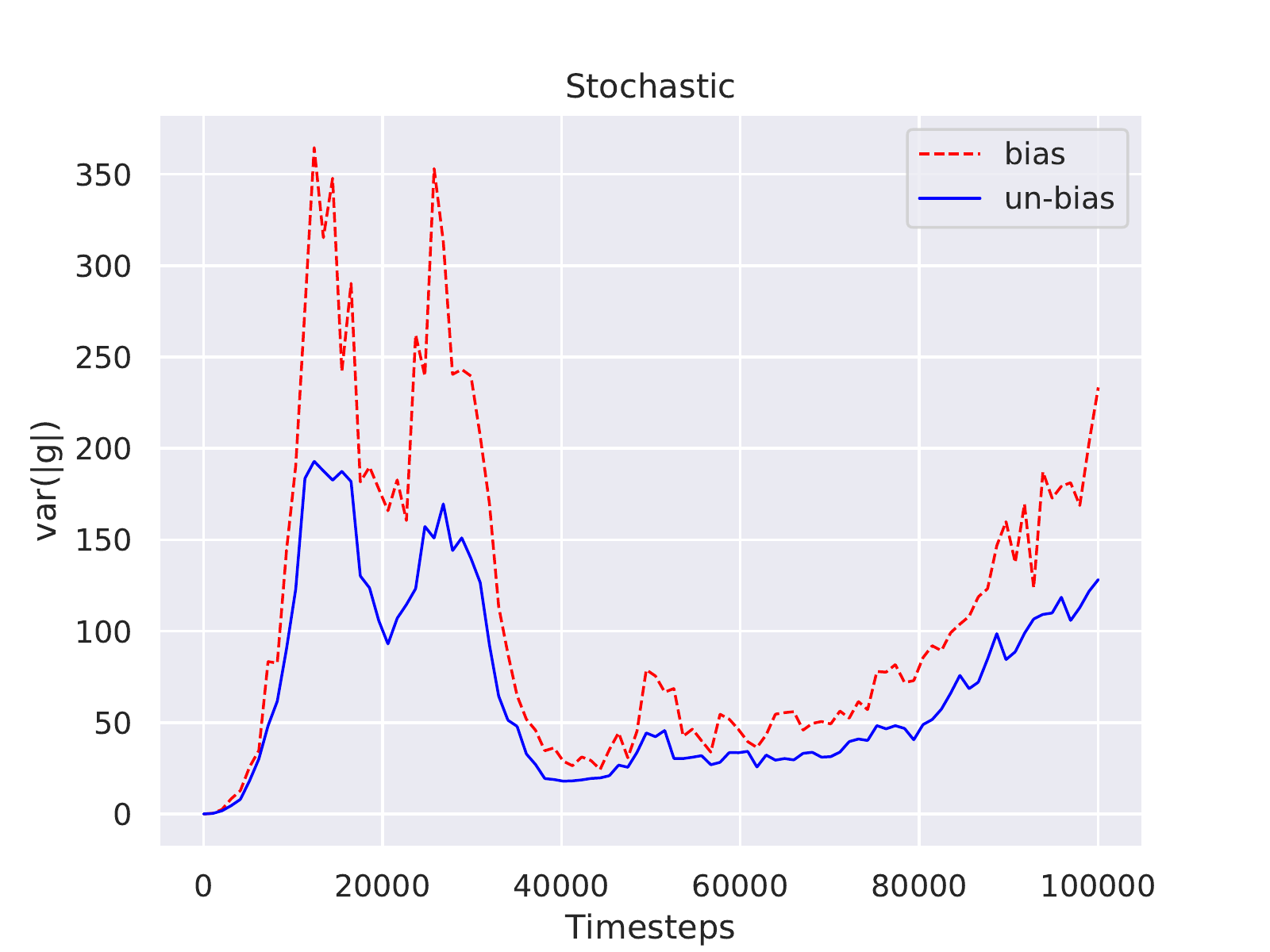}
\caption{Two different gradient estimations in deterministic and stochastic dynamics.}
\label{fig:if-biasgrad-det}
\end{figure}

%
\item  \textbf{The Effect of Multi-steps.} In temporal-difference algorithms and cases where the reward function is highly sparse, multi-step setting is commonly utilized for better bias-variance trade off \cite{kearns2000bias,sutton1998reinforcement} and reward condensing, respectively. The optimal Bellman equation in \eqref{optimal-bellman} is the one-step case and it can be extended into multi-step version, as well as the corresponding LP formulations \eqref{eq:rllp-primal} and \eqref{eq:rllp-dual}. Therefore, SCAL can be applied to the multi-step setting by modifying the one-step reward $r_i$ in transition tuple $(s_i,a_i,r_i,s'_i)$ into
\be
r_i = \sum_{j=0}^{l-1}\gamma^{j}r(s_{i+j},a_{i+j}),
\ee
where $l$ is the length of looking ahead, and $s'_i$ is the state after $l$ steps transition from $s_i$. The parameter $l$ determines how far into the future the algorithm can reach. The bias of the Bellman residual is reduced as the lookahead length $l$ increases, while the variance is also enlarged unexpectedly. We test different values of $l$ in two robotic locomotion control problems, HalfCheetah and Ant, to investigate the influence on the performance. We plot the results in Fig \ref{fig:multistep}. It reveals that slightly increasing the lookahead length is quite efficient in promoting the performance, while looking too far ahead induces unexpected variance which may suppresses the learning of the algorithm. 
\begin{figure}[tbhp]
\centering
\includegraphics[scale = 0.36]{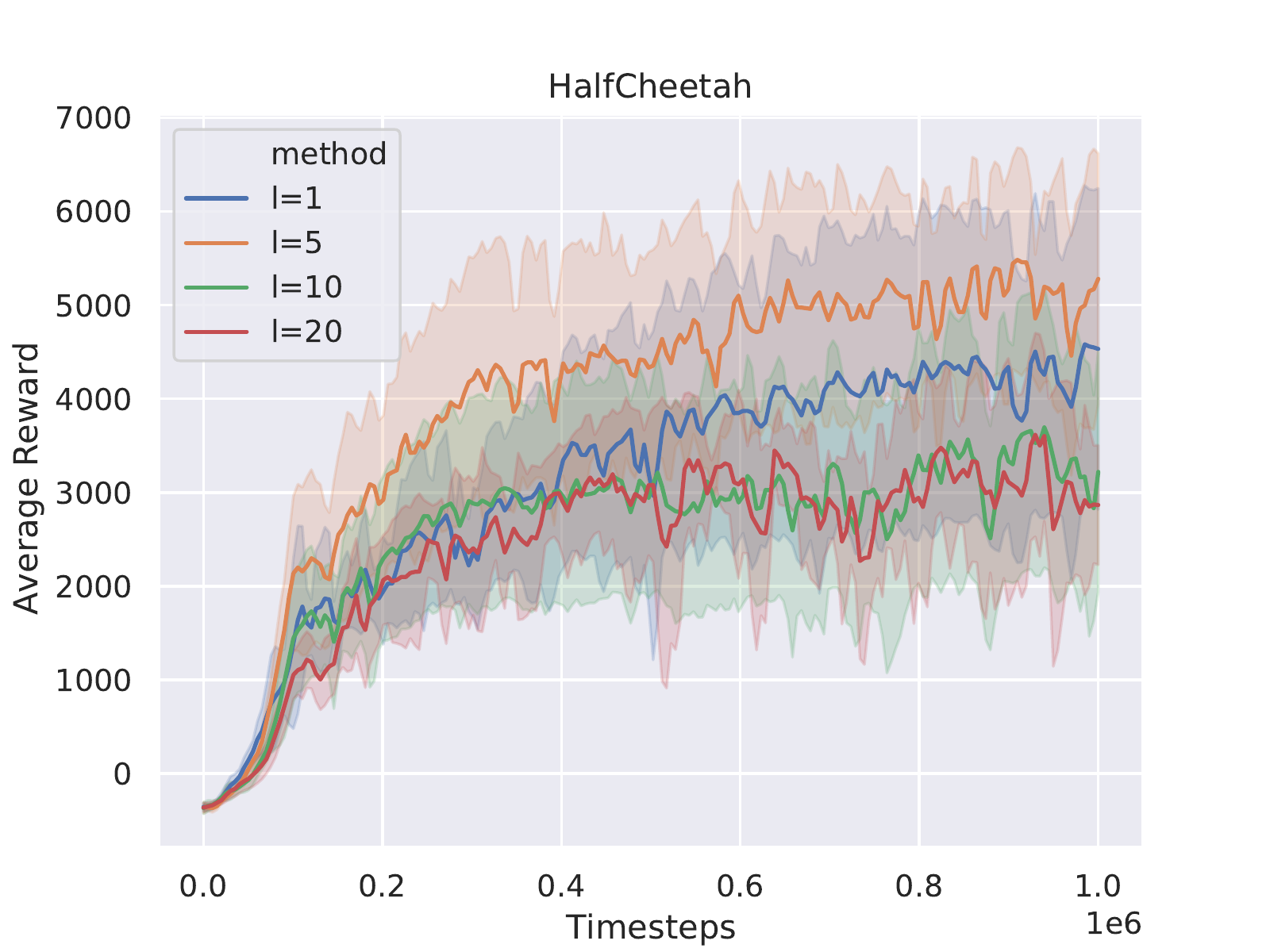}
\includegraphics[scale = 0.36]{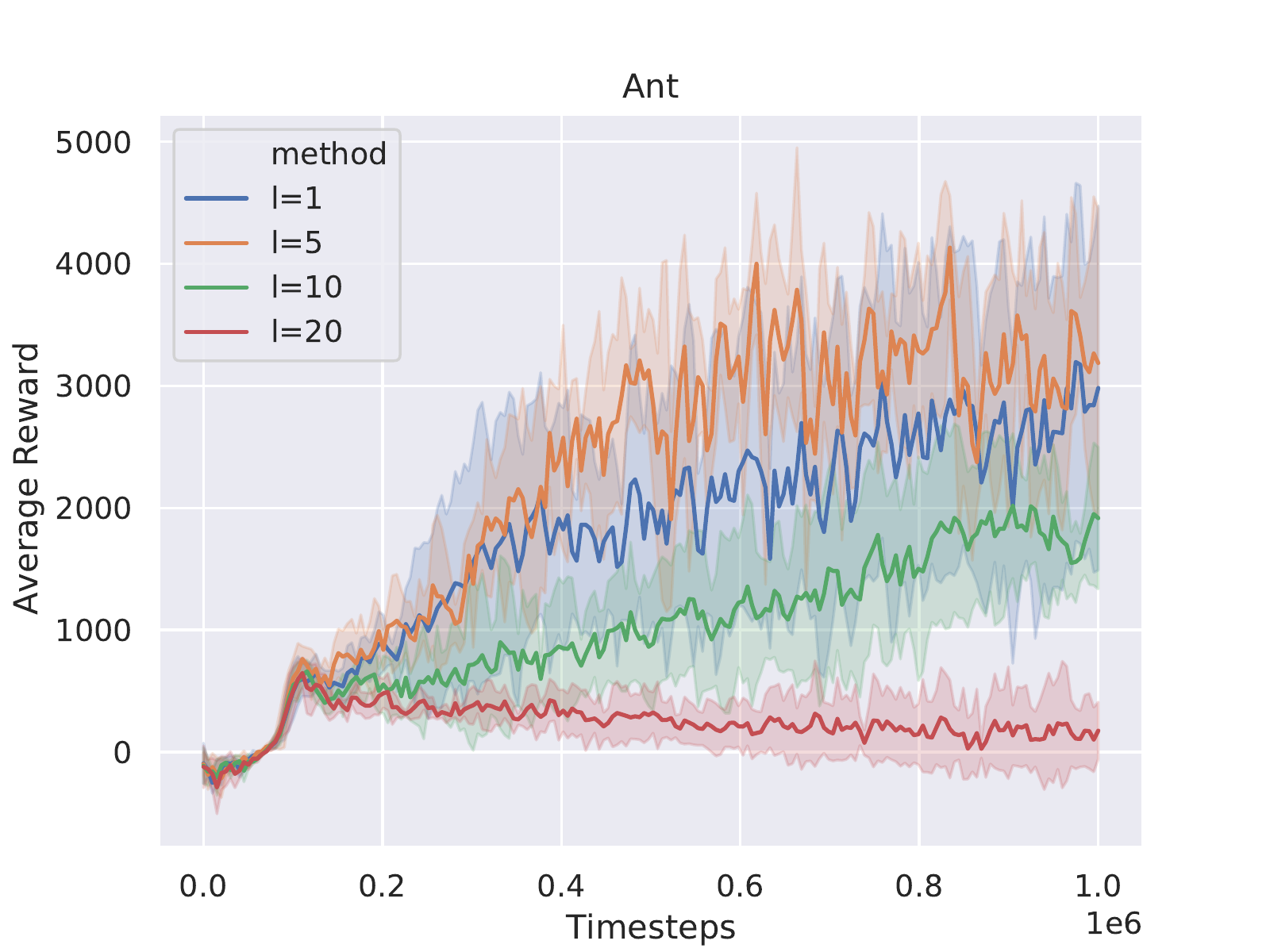}
\caption{The comparison among different lookahead length.}
\label{fig:multistep}
\end{figure}


\end{itemize}

\subsection{Inventory Control}
We consider the problem of day-to-day control of an inventory of a fixed maximum size $M$ in the face of uncertain demand\footnote{Example 1 in https://sites.ualberta.ca/~szepesva/papers/RLAlgsInMDPs.pdf}. It is a single agent domain featuring discrete state and action space. In the evening of day $t$, the agent must decide about the quantity $a_t\in\{0,1,2,...,M\}$ to be ordered for the next day $t+1$, in which $s_t\in\{0,1,2,...,M\}$ is the current size of the inventory. During the day $t+1$, some stochastic demand $d_{t+1}$ arrives, and the next inventory is 
\be
s_{t+1} = [\min((s_t+a_t),M) - d_{t+1}]_+,
\ee
in which $\{d_t\}$ is a sequence of independent and identically distributed (i.i.d.) integer-valued random variables.
The revenue on day $t+1$ is defined as 
\be
r_{t+1} = -K \cdot1_{a_t>0}-c\cdot [\min((s_t+a_t),M) - s_{t}]_+ - h\cdot s_t + p\cdot[\min((s_t+a_t),M) - s_{t+1}]_+,
\ee
where $K>0$ is a fixed entry cost, $c>0$ is the unit cost of purchasing, $h>0$ is the unit cost of holding an inventory and $p>h$ is the unit cost of selling.

In our experiments, we consider the setting as follows: $M=100$, $K=5$, $c=2$, $h=2$, $p=3$ and the Poisson distribution $\mathcal{P}(8)$ is considered for the demand variables. The lines in Fig \ref{fig:inventorycontrol} illustrate the average cumulative return over the last simulated 10 days. The result shows that SCAL reaches a stable solution with quite less interactions than other two state-of-the-art RL algorithms, TRPO \cite{schulman2015trust} and PPO \cite{schulman2017proximal}.
\begin{figure}[tbhp]
\centering
\includegraphics[scale = 0.40]{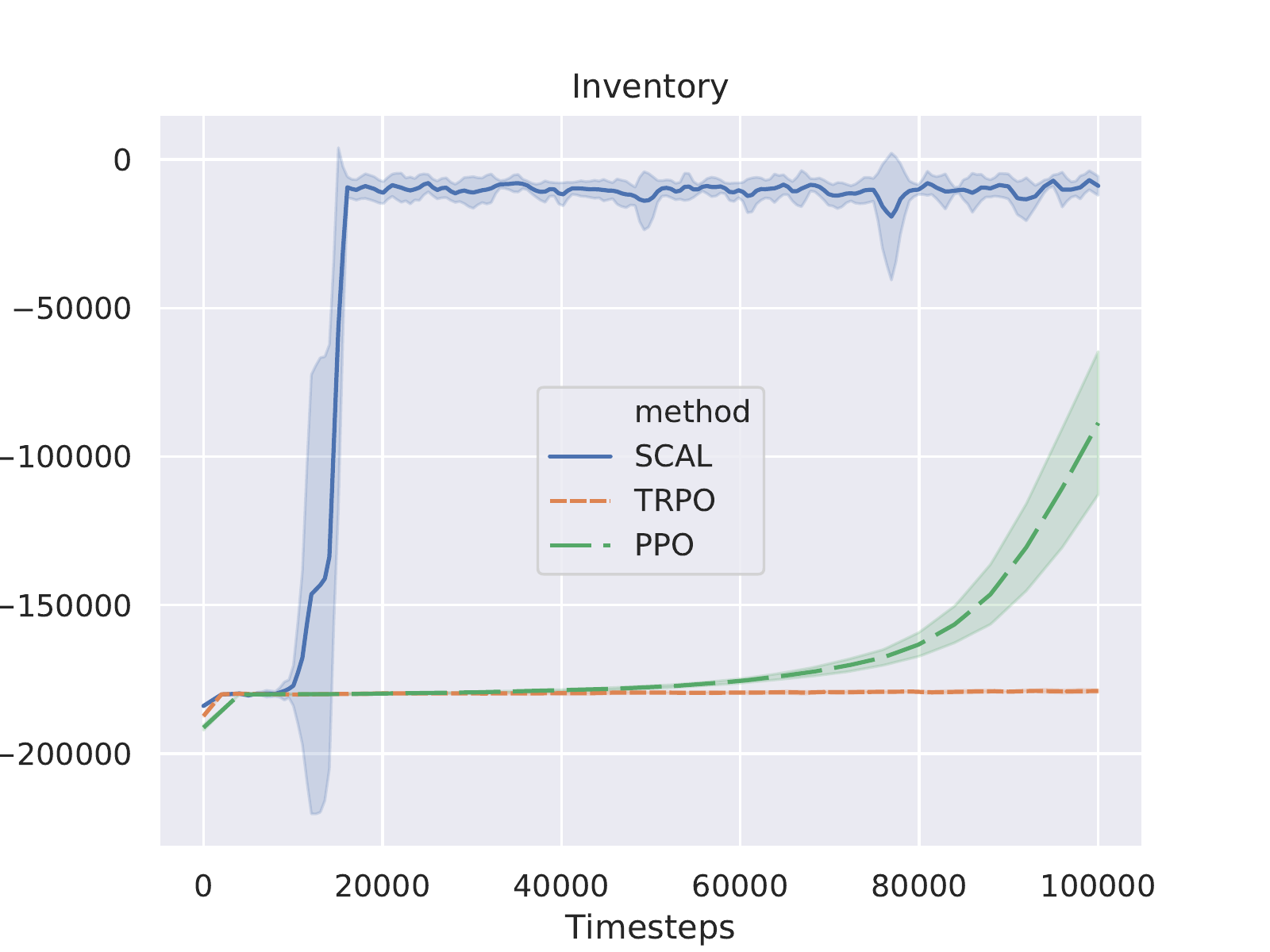}
\caption{The Inventory Control.}
\label{fig:inventorycontrol}
\end{figure}

\subsection{Comparisons in Discrete Environments}
To demonstrate the learning ability, we compare SCAL with several popular RL algorithms, including DQN, two policy gradient methods TRPO and PPO on two discrete control problems. We plot the results in Fig \ref{fig:comparisons-discrete}. With the same number of interactions, our algorithm achieves significant better performance in learning capability and stability. Note that TRPO and PPO are only applicable in on-policy setting which has a poor sample efficiency, while both our algorithm and DQN are performed in an off-policy manner.
\begin{figure}[tbhp]
\includegraphics[scale = 0.36]{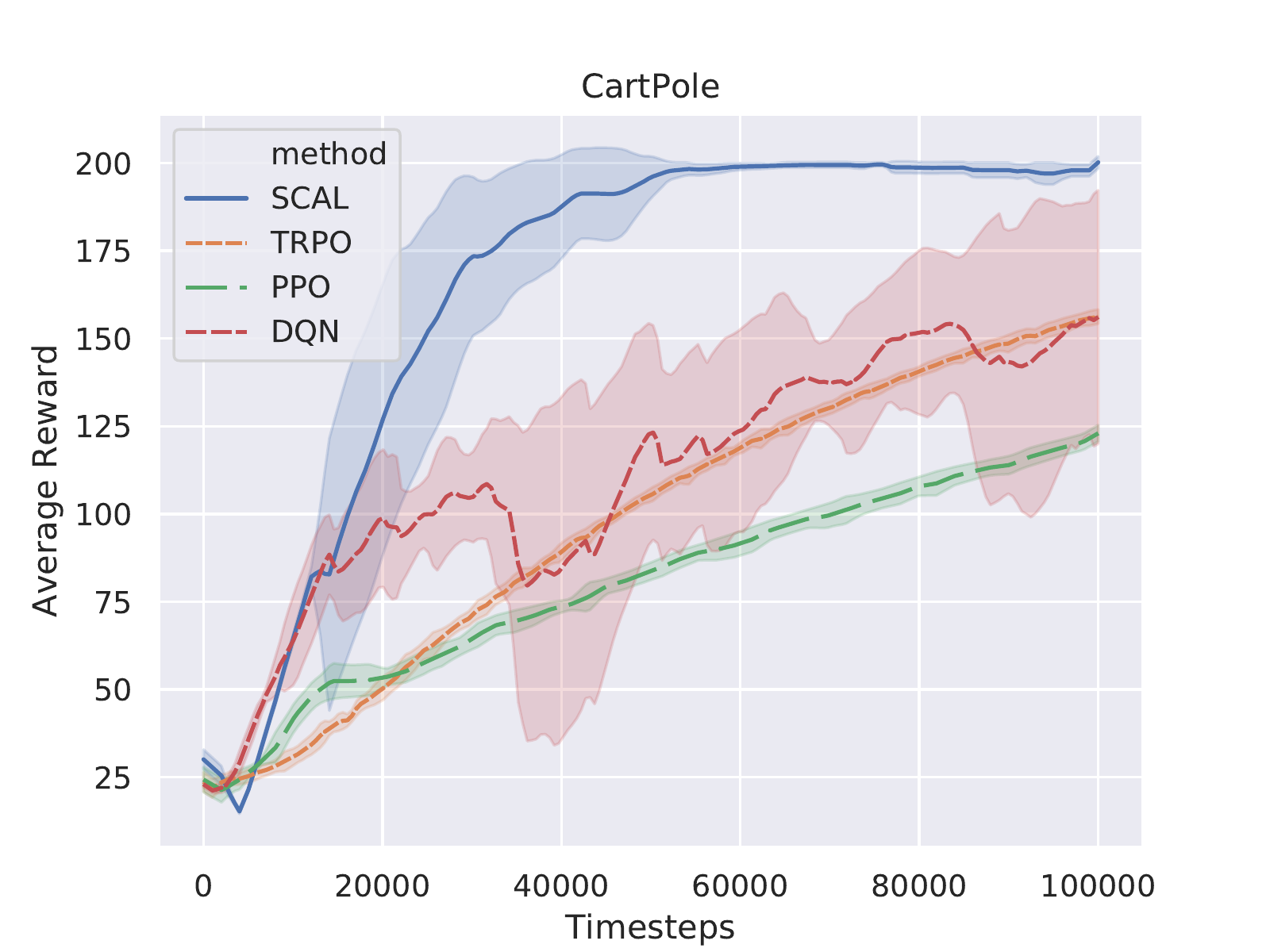}
\includegraphics[scale = 0.36]{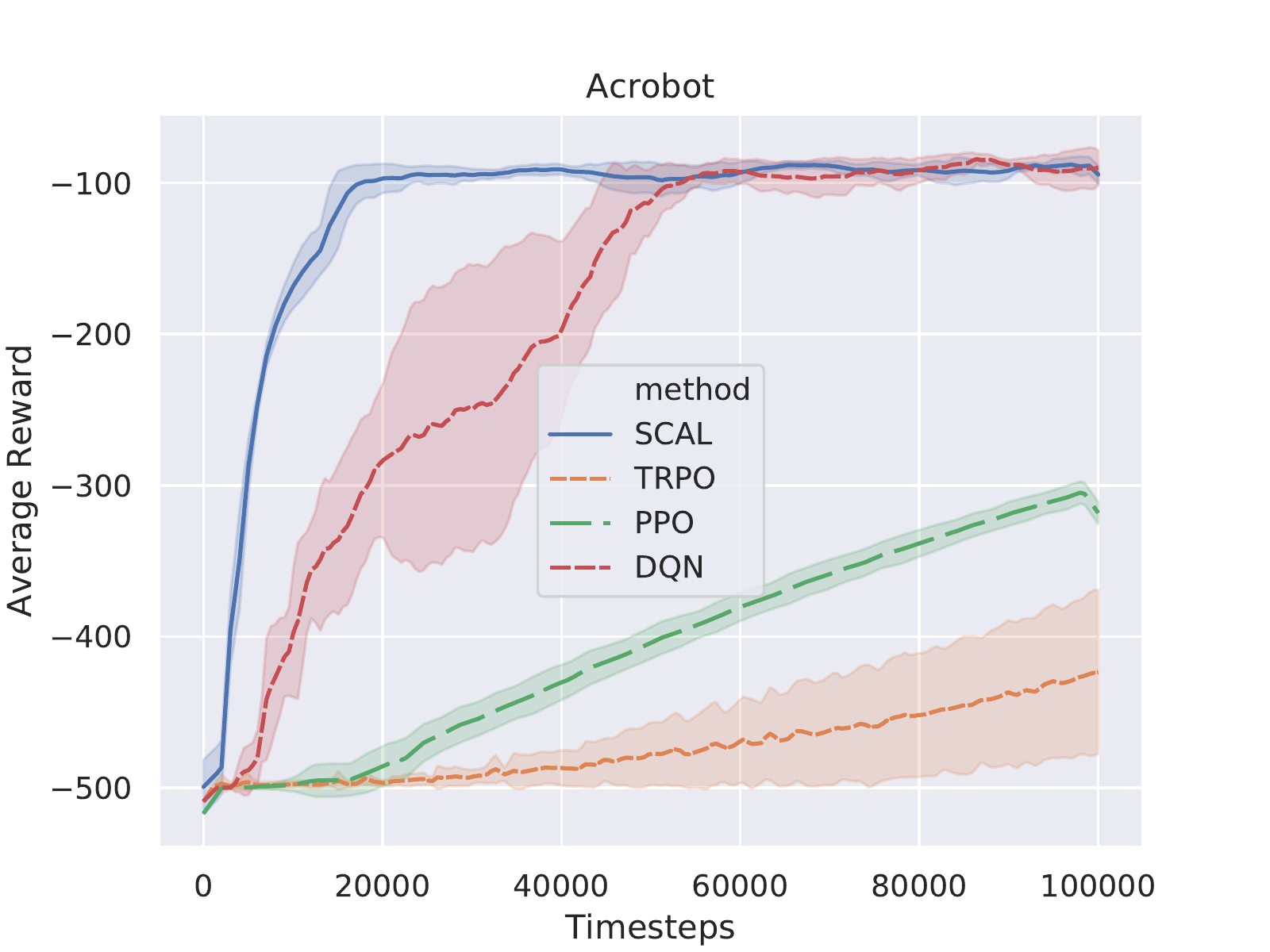}
\caption{Comparisons among discrete environments.}
\label{fig:comparisons-discrete}
\end{figure}

\subsection{Comparisons in Continuous Control Problems}
 As DQN is unavailable in continuous environments, we just compare SCAL with TRPO and PPO on several control tasks provided from OpenAI Gym benchmark \cite{brockman2016openai} using the MuJoCo simulator \cite{todorov2012mujoco}. The different dynamic properties and dimensions of these environments provide meaningful benchmarks to make a comparison for the algorithms. In Fig \ref{fig:comparisons}, our method attains higher average reward with the same number of interactions than the other two state-of-the-art algorithms. 
\begin{figure}[tbhp]
\includegraphics[scale = 0.36]{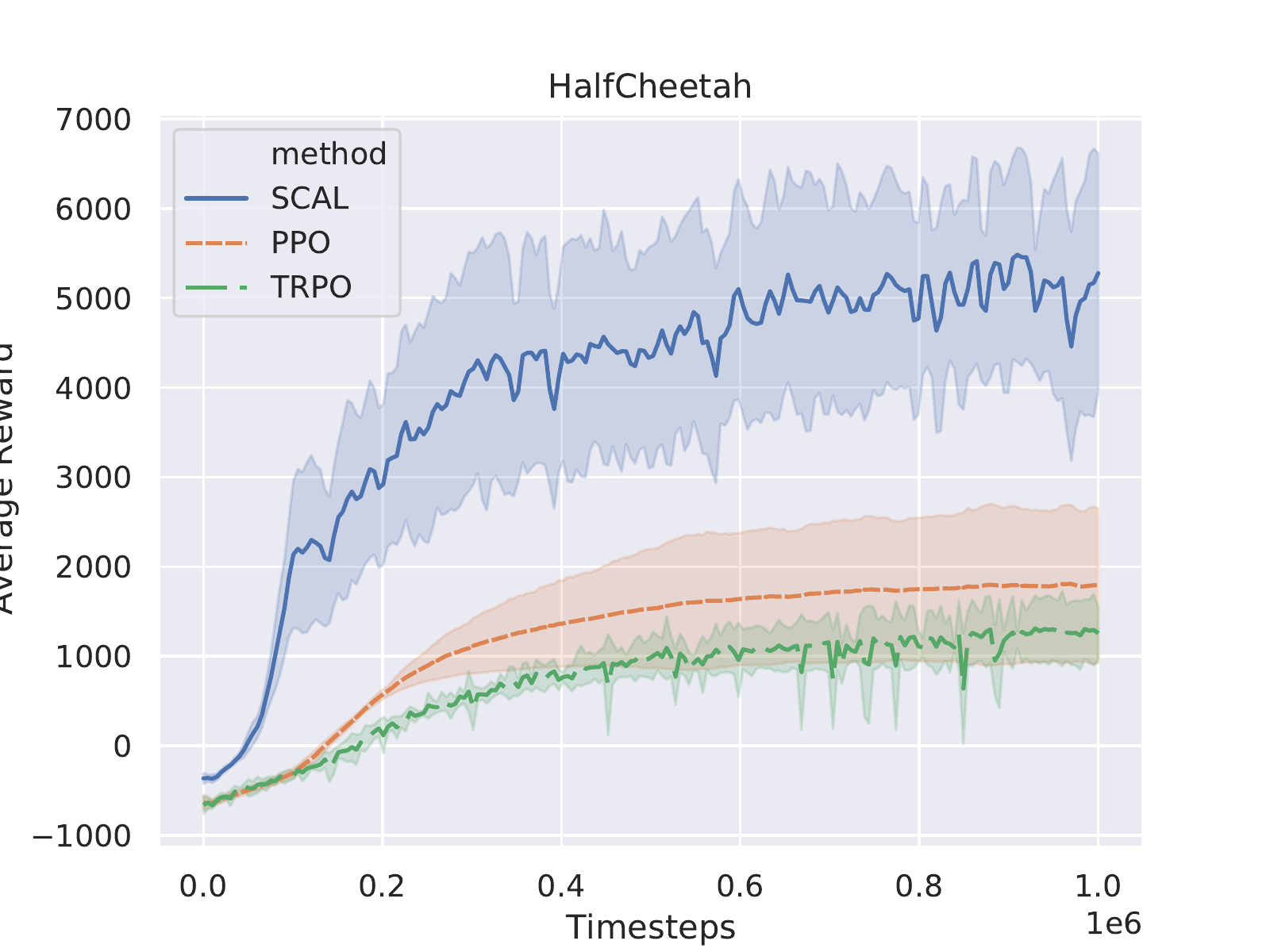}
\includegraphics[scale = 0.36]{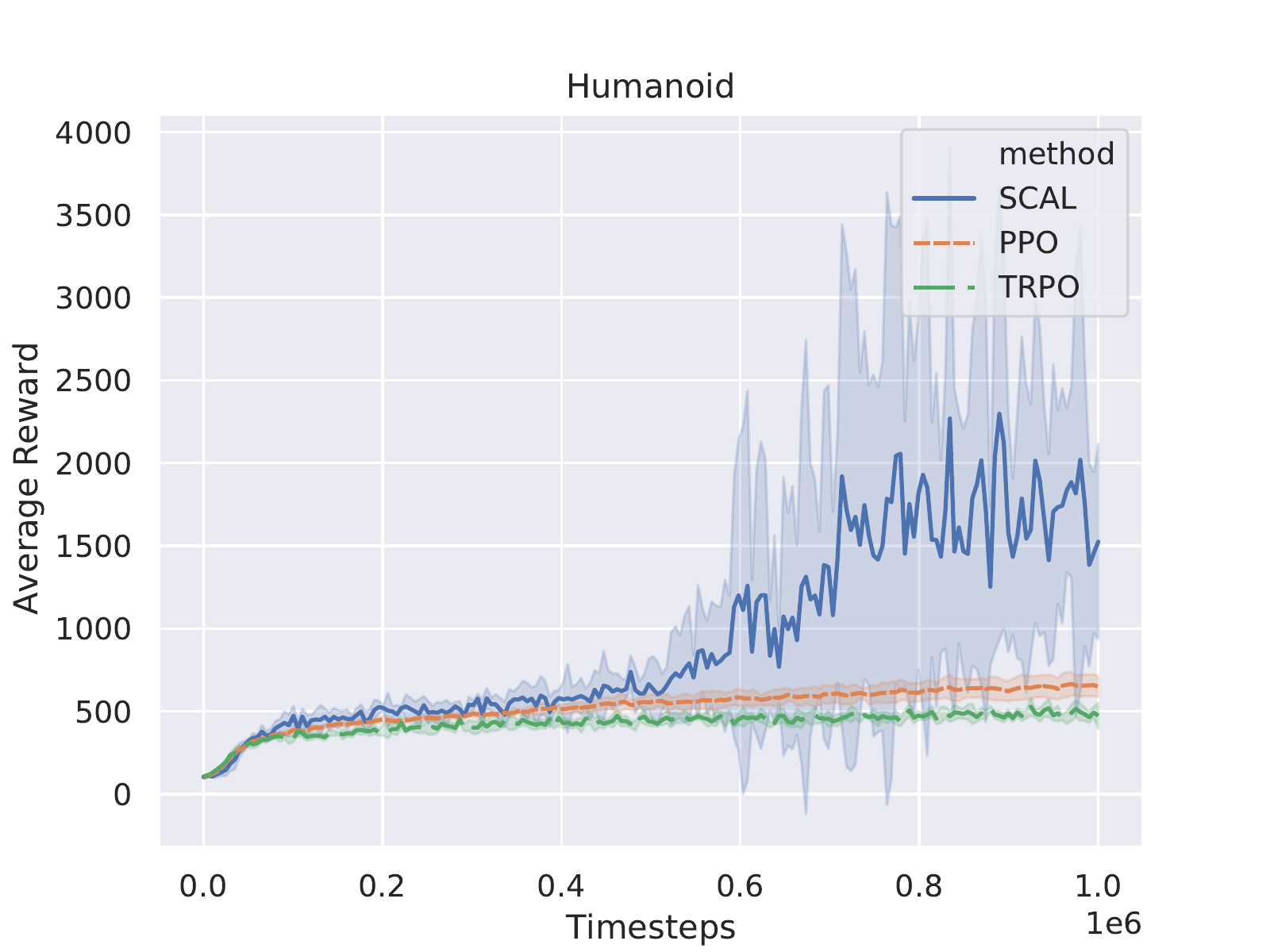}\\
\includegraphics[scale = 0.36]{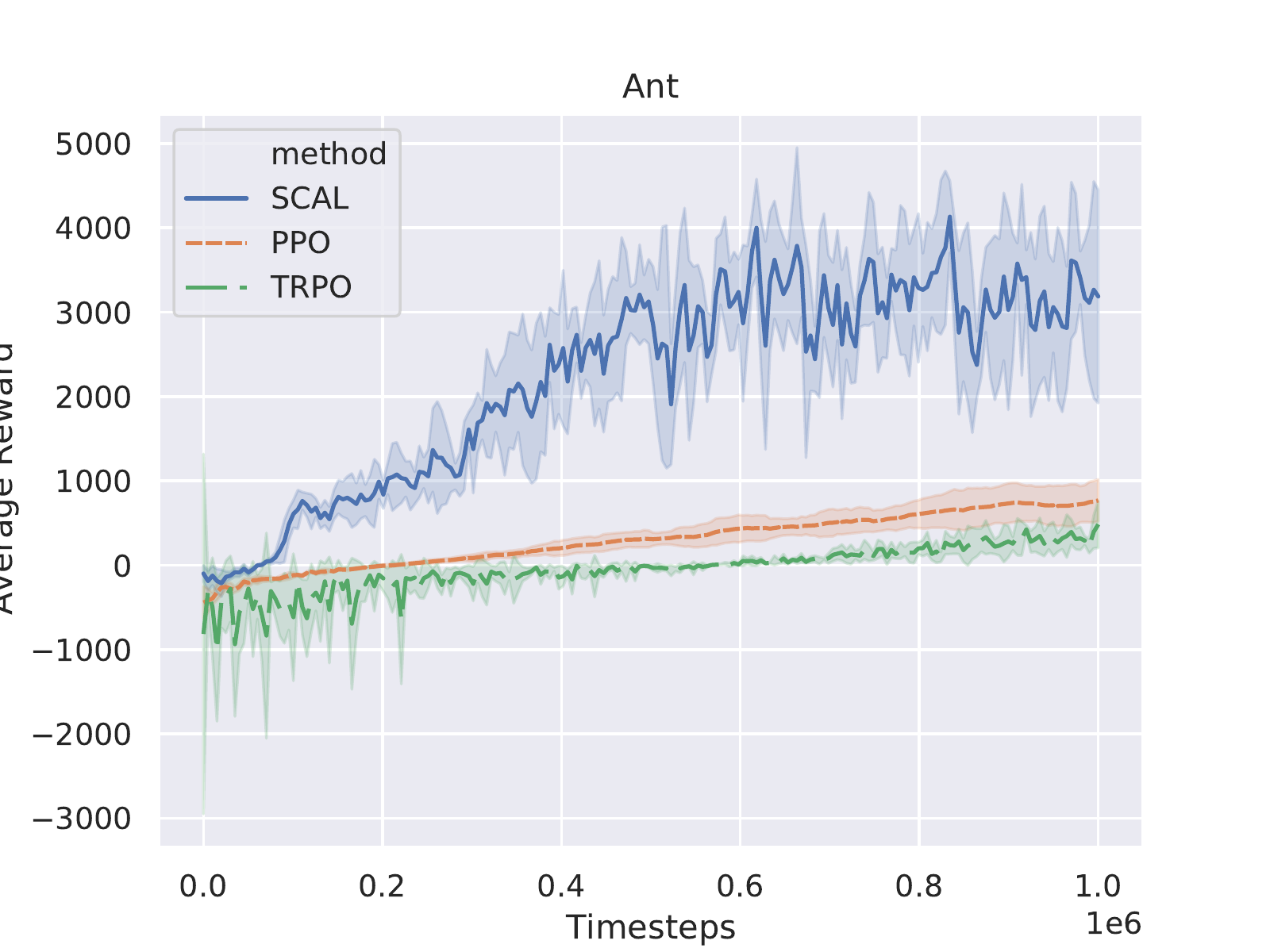}
\includegraphics[scale = 0.36]{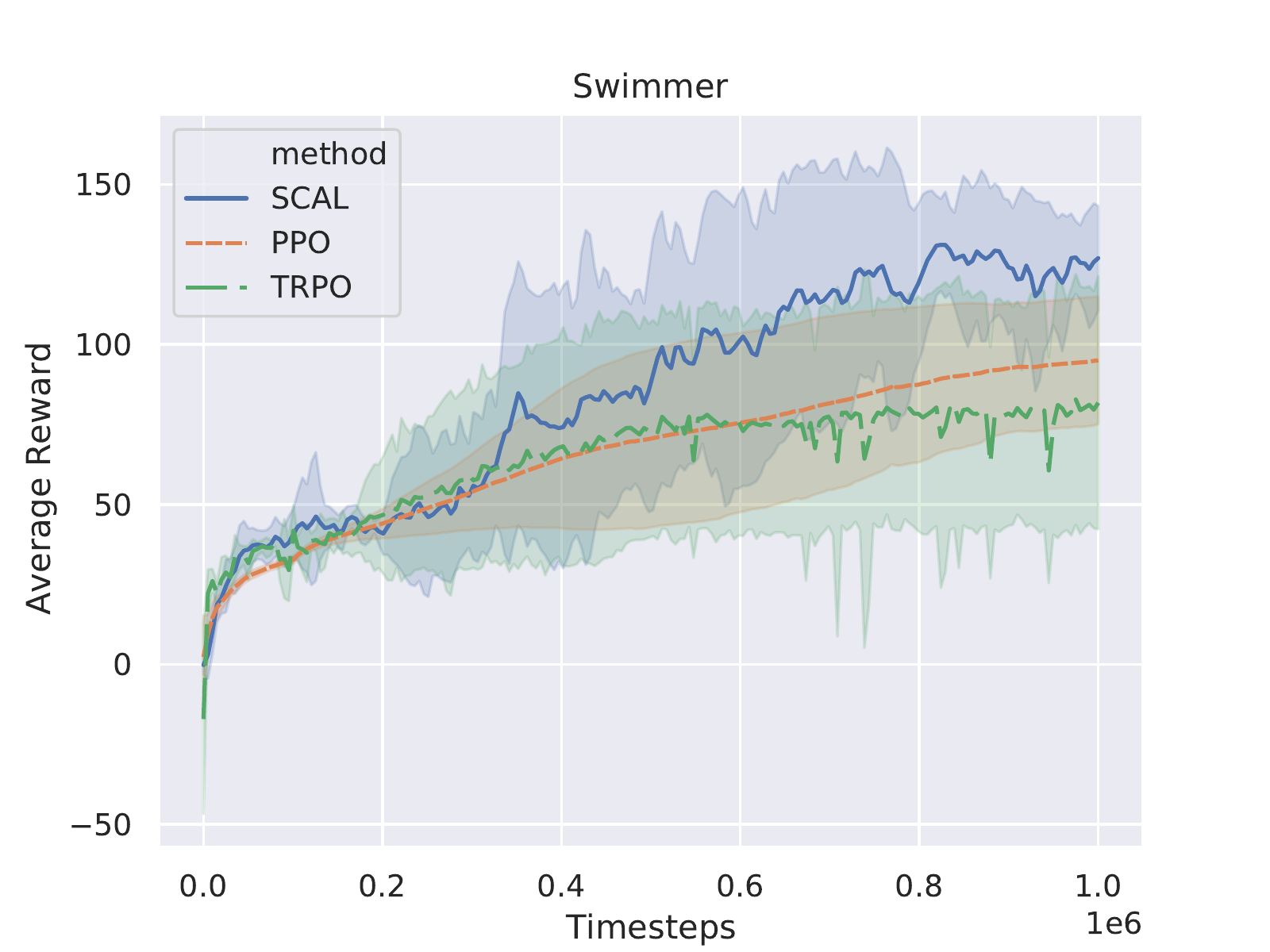}
\caption{Comparisons among state-of-the-art algorithms.}
\label{fig:comparisons}
\end{figure}

\section{Conclusion}
With the help of the algorithmic property of ALM, we overcome the difficulties in minimizing the weighted augmented Lagrangian function by taking advantage of the multipliers. Our solution provides an opportunity to incorporate the subproblems in ALM into a single constrained problem. To develop a practical optimization method, we penalize the constraints to formulate a quadratic penalty problem and employ semi-gradient for the value function, so that a particular stochastic gradient optimization is easily to be performed for the objective function. The ablation study shows that the double-sampling solution significantly reduces the variance of the gradient estimation in environments with stochastic transition probability, and the algorithm can be easily extended into multi-step version with better performance. The comparisons with other state-of-the-art RL algorithms also demonstrate that our method has competitive learning ability in multiple environments.


%
%


\section*{Acknowledgement}
The computational results were obtained at GPUs supported by the National Engineering Laboratory for Big Data Analysis and Applications and the High-performance Computing Platform of Peking University.

\bibliographystyle{siamplain}
\bibliography{article.bbl}
\newpage

\end{document}